\documentclass{article}
\usepackage{graphicx}
\usepackage{latexsym}
\usepackage{amsmath}
\usepackage{verbatim}
\usepackage{amssymb}
\usepackage{mathrsfs}
\usepackage{amsmath,amssymb}
\DeclareMathAlphabet      {\mathbf}{OT1}{cmr}{bx}{n}
\DeclareFontFamily{OT1}{pzc}{}
\DeclareFontShape{OT1}{pzc}{m}{it}%
{<-> s * [1.15] pzcmi7t}{}
\DeclareMathAlphabet{\mathpzc}{OT1}{pzc}{m}{it}

\usepackage{ dsfont }
\usepackage{ mathrsfs }
\usepackage{enumitem}
\usepackage{amsbsy }
\usepackage{graphics}

\usepackage{amsthm}
\newtheorem{thm}{Theorem}
\newtheorem{theorem}{Theorem}[section]

\newtheorem{lemma}[theorem]{Lemma}
\newtheorem{corollary}[theorem]{Corollary}

\newtheorem{definition}[theorem]{Definition}

\newtheorem{prop}[theorem]{Proposition}

\newtheorem{remark}{Remark}[section]

\usepackage{romannum}
\usepackage[toc,page,title,titletoc,header]{appendix}
\usepackage{hyperref}

\usepackage[all,cmtip]{xy}

\usepackage{lipsum}

\usepackage[T1]{fontenc}
\usepackage[utf8]{inputenc}
\usepackage{authblk}

\def\ack{\section*{Acknowledgements}%
  \addtocontents{toc}{\protect\vspace{6pt}}%
  \addcontentsline{toc}{section}{Acknowledgements}%
}
\usepackage{abstract}

\usepackage{titling}
\usepackage{blindtext}

\begin{document}  \large
\pagenumbering{arabic}

\title{Cobordism maps on periodic Floer homology induced by elementary Lefschetz fibrations }
\author{{ Guanheng Chen}}
\date{}
%\begin{titlingpage}
\maketitle
%Department of Mathematics, University of Adelaide \\
%\verb| E-mail adress: guanheng.chen@adelaide.edu.au |
\thispagestyle{empty}
\begin{abstract}    % type your abstract below
Periodic Floer homology (PFH) is a Gromov--Floer type invariant for fibered three--manifolds with Hamiltonian structures.  The cobordism maps on periodic Floer homology induced by symplectic cobordisms are currently only defined indirectly by using Seiberg--Witten theory. In this paper, we investigate the cobordism maps induced by a class of symplectic cobordisms constructed by P. Seidel, called elementary Lefschetz fibrations. We define the PFH cobordism maps induced  by elementary Lefschetz fibrations in terms of  holomorphic curves.  Moreover,   we compute  these  maps for some cases.
\end{abstract}
%\end{titlingpage}
%\blinddocument
%clearpage

\section{Introduction and Main results}
Let $\pi: Y \to S^1$ be  a surface fibration  over the circle together with a closed  fiberwise non--degenerate $2$--form    over $Y$. Such a 2--form is called  an \emph{admissible 2--form} or  a \emph{Hamiltonian structure}.   Assume that the fiber of $\pi$ is a connected oriented surface, possibly with boundary.
%The $2$-form $\omega$ gives a splitting of $TY=TY^{vert}\oplus TY^{hor}$, where $TY^{vert}= ker \pi_*$ and $TY^{hor}$ is the orthogonal complement of $TY^{vert}$ with respect to $\omega$.    The horizontal lift of the vector field of $\partial_t$ is  called Reeb vector field $R$, where  $\partial_t$  is coordinate vector field of $S^1$.  A periodic orbit $\gamma$ of $(Y, \pi, \omega)$ is the periodic integral curve of $R$.
%\begin{equation*}
%\gamma: \mathbb{R}/{d \mathbb{Z}} \to Y, \partial_t \gamma =R \circ \gamma(t),
%\end{equation*}
%where $d \ge 0$ is called period or degree of $\gamma$.
%Fix a homology class $\Gamma \in H_1(Y, \mathbb{Z})$,
 Under the above setup, M. Hutchings introduces a Gromov--Floer type invariant called  \emph{periodic Floer homology} (abbreviated as PFH) $PFH_*(Y, \omega)$ \cite{H1}. The definition will be reviewed later. Y--J. Lee and C. H. Taubes show that PFH is isomorphic to a version of Seiberg--Witten cohomology \cite{LT}.

Given a Lefschetz fibration $\pi_X: W \to B$ together with a closed fiberwise nondegenerate $2$--form $\omega_X$ such that $\partial W=Y$ and $\omega_W \vert_Y =\omega$,  it is expected that the Lefschetz fibration induces a homomorphism  $PFH(W, \omega_W): PFH(Y, \omega) \to  \Lambda$, where  $ \Lambda$ is a local coefficient system. Such a homomorphism is called \emph{a cobordism map}.   We call the triple $(W, \pi_W, \omega_W)$ a \emph {fiberwise symplectic cobordism}.  The 2--form $\omega_W$ is called \emph {a fiberwise symplectic form} or \emph {an admissible 2--form.}  Note that $\omega_W$ is not necessary symplectic.  But we can construct a natural symplectic form $\Omega_W$ by taking $\Omega_W=\omega_W + \pi^*_W \omega_B$, where $\omega_B$ is a large volume form of  $B$.

The purpose  of this paper is to understand the cobordism maps induced by a special class of fiberwise symplectic cobordisms, called  \emph{elementary Lefschetz fibrations},  in terms of holomorphic curves.     Such Lefschetz fibrations are constructed by  Seidel \cite{PS}, \cite{PS1}. We  use  $(X, \pi_X, \omega_X)$ to denote the  elementary Lefschetz fibration throughout.

\paragraph{Elementary Lefschetz fibrations} Roughly speaking, an elementary Lefschetz fibration $(X, \pi_X, \omega_X)$ consists  of  a Lefschetz fibration  $(X, \omega_X)$ over a disk with exactly one critical point together with  a fiberwise symplectic form $\omega_X$.  
Let us  briefly review    its construction. More details are also given in  Section \ref{section3}. %More details will be explained in Section 2.
Start  with the local model
\begin{align} \label{eq18}
\pi:  \ &\mathbb{C}_{\textbf{x}=(x_1, x_2)}^2  \to \mathbb{C}  \ \ \textbf{x} \to x_1^2 +x_2^2,
\end{align}
 we first construct an exact  Lefschetz fibration $(E, \pi_E)$ by a suitable cutting of  (\ref{eq18}). More precisely,  for $\lambda, \delta>0$, define $E=\{ \textbf{x} \in \mathbb{C}^2 \vert  |\textbf{x}|^4-|\pi(\textbf{x})|^2 \le 4 \lambda^2, \ |\pi(\textbf{x})| \le \delta\}$ and $\pi_E=\pi \vert_E$. Then $E$ is a Lefschetz fibration over a disk $D_{\delta}$ with radius less than or equal to $\delta$. $(E, \pi_E)$ has a unique critical point at the origin.  A regular fiber of $\pi_E$ is a cylinder of length $2\lambda$. Note that $E$ is a manifold with corners. The boundary of $E$ can be divided into the vertical boundary $\partial_v E=\pi^{-1}(\partial D_{\delta})$ and the horizontal boundary $\partial_h E =  \sqcup_{z \in D_{\delta}} \partial \pi^{-1}(z)$.
Topologically, $(X, \pi_X)$ is obtained by gluing $(E, \pi_E)$ with a trivial fibration $S \times D_{\delta}$  over a disk along their horizontal boundaries, i.e., $X=E \cup_{\partial \pi^{-1}(z) \sim \{z\} \times \partial S} (S \times D_{\delta})$, where $S$ is an oriented compact surface with two boundary components.

A  candidate of the  fiberwise symplectic form over $E$ is the restriction of the standard symplectic form  $\omega_{\mathbb{C}^2 } \vert_E$. For the purpose of gluing $\omega_{\mathbb{C}^2 } \vert_E$ with an area form $\omega_S$ of $S$, we need to modify $\omega_{\mathbb{C}^2 } \vert_E$ such that it is ``trivial near $\partial_h E$''.    $\omega_X$ is the result of  this gluing. We will explain more details about this point in Section \ref{section3}.

Note that the construction of $(X, \pi_X, \omega_X)$ depends on many choices such as $\lambda, \delta, \omega_S$, etc. But we still prefer to call it `the elementary  Lefschetz fibration' because they are isotropic to each other in the sense:  Given two elementary Lefschetz fibrations $(X, \pi_X, \omega_X)$ and  $(X', \pi_{X'}, \omega'_X)$ constructed by different data, then  there is a diffeomorphism  $F:X \to X'$ preserving the Lefschetz fibration structure such that %$F^* \omega_X' \vert_{fiber} = \omega_X \vert_{fiber}$ and
$[F^* \omega'_X -\omega_X]=0 \in H^2(X, \mathbb{R})$.  Also, the Floer homology and its cobordism maps under consideration should not depend on this  deformation.

\paragraph{Results} The main result of this paper is to construct the cobordism maps induced by the elementary Lefschetz fibrations by using holomorphic curves. Moreover, we compute them for some cases. \emph{We assume that the fiber $F$ of $X$  is a closed surface with genus $g(F) \ge 2$  and the vanishing cycle is non--separating; unless otherwise stated.}
\begin{thm} \label{thm3}
Let $(X, \pi_X, \omega_X)$ be an  elementary Lefschetz fibration. Fix an integer  $Q $. Assume that  $\int_F \omega_X \ge Q+1$ and $Q \ne g(F)-1$.  %The fiber of $(X, \pi_X)$ is denoted by $F$.
\begin{enumerate} [label=\textbf{\Alph*.}]
\item \label{A}
Then for  a generic  $\Omega_X$--tame almost complex structure $J$ which is sufficiently close to $  \mathcal{J}_h(X, \omega_X)$ (see Definition \ref{def2}),   we have a well--defined homomorphism
\begin{equation*}
PFH(X, \omega_X)_J : \mathds{A}(X)\otimes PFH_*(Y, \omega, Q)  \to \mathbb{Z},
\end{equation*}
where $\mathds{A}(X)=[U]\otimes \Lambda^*H_1(X, \mathbb{Z})$ and $U$ is the $U$--map (see \cite{H3}).
The homomorphism is defined by counting $J$--holomorphic curves.% without using Seiberg--Witten  theory.
\item  \label{B}
 Moreover, we have:
\begin{itemize}
\item
If $Q > g(F)-1$,  then
\begin{equation*}
PFC(X, \omega_X)_J(  (\Pi_a e_a) e_0^{m_0} e_1^{m_1})=1
\end{equation*}
and $PFC(X, \omega_X)_J$ maps the other  ECH generators   to zero.  Here  $a$  are     critical points of a Morse function  $f_S$ with $\nabla^2 f_S(a) >0$.  Each  $e_a$ is  the periodic orbit corresponding to $a$,    and
  $\{e_i\}_{i=0, 1}$ are the only two degree $1$ elliptic periodic orbits of the Dehn twist. For the precise definition, please refer to   Section \ref{section2}.

\item
If $g(F)-1\ge 2Q$, then
\begin{equation*}
PFH(X, \omega_X)_J(   e^m)=1
\end{equation*}
and maps the  other  generators  of  $PFH_*(Y, \omega, Q)$ to zero.  In this case, $e_0$ and $e_1$  are homologous and $e$ is their homology class.
\end{itemize}
\end{enumerate}
\end{thm}
It is not difficult to see that the cobordism map $PFH(X, \omega_X)_J$ can  split into
\begin{equation*}
PFH(X, \omega_X)_J=\sum_{\Gamma_X \in H_2(X, \partial X, \mathbb{Z})} PFH(X, \omega_X, \Gamma_X)_J.
\end{equation*}
Some of these components  are invariant under the   blow--up (see  Lemma 5.2 of \cite{GHC}). Thus  we have the following corollary.
\begin{corollary}
Let $\{x_i\}_{i=1}^k \subset X$ be a  finite set of points such that $\pi_X(x_i) \ne \pi_X(x_j)$ for $i \ne j$ and $\pi_X(x_i) \ne 0$ for all  $i$. Let $(\pi_X' : X' \to D, \omega_{X'})$ be the blow--up  of $X$ at $\{x_i\}_{i=1}^k$.  For any $\Gamma_X \in H_2(X, \partial{X}, \mathbb{Z})$ regarded as a homology class in $X'$, then  $PFH(X', \omega_{X'}, \Gamma_X)_{J'} $ satisfies the same conclusions in Theorem \ref{thm3}.
\end{corollary}

\begin{remark}
 In general, neither PFH nor its cobordism maps can be defined with  $\mathbb{Z}_2$-- or $\mathbb{Z}$--coefficients.  One needs certain  monotonicity  assumptions or introducing the local coefficient system (see \cite{H1}, \cite{LT}). But all the objects we are considering satisfy the  monotonicity  properties  if $Q\ne g(F)-1$. For details please see  Lemma 5.1 of \cite{H1} and Remark \ref{rmk1}.  We   use $\mathbb{Z}$--coefficient throughout. %In our case, $(Y, \pi, \omega)$ is a mapping torus of Dehn twist. By Lemma 5.1 of \cite{H1}, it  satisfies the  monotonicity  condition.  We   use $\mathbb{Z}$--coefficient throughout. 
\end{remark}

\begin{remark}
	%The assumption $Q\ne g(F)-1$ in Theorem \ref{thm3} in fact is not necessary.  
	The assumption $Q\ne g(F)-1$ is only used to guarantee  that the   monotonicity properties are  true and it plays no role elsewhere.  Therefore,  the statement of part \ref{A} in Theorem \ref{thm3} is still true if we replace the $\mathbb{Z}$--coefficient  by a local coefficient system.   
\end{remark}

The following two remarks concern the relations between  $PFH(X, \omega_X)_J$ and the cobordism maps  $HP_{sw}(X, \Omega_X)$ defined in \cite{GHC}.  The construction of  $HP_{sw}(X, \Omega_X)$  is parallel to Hutchings and Taubes's results for embedded contact homology \cite{HT}. %The subscript ``sw"  is used to emphasize that $HP_{sw}(X, \Omega_X)$ is defined by Seiberg--Witten theory.
It relies heavily on the isomorphism “PFH=SWF” \cite{LT} and the Seiberg--Witten theory. 
\begin{remark}
Suppose that $Q>g(F) -1$ and the cobordism maps $PFH(X, \omega_X)_J$  are defined by using almost complex structures compatible with the symplectic form. (\cite{GHC} calls them cobordism--admissible.)  Then  we can show that    $PFH(X, \omega_X)_J$ agrees with the one $HP_{sw}(X, \Omega_X)$ defined in \cite{GHC}. In particular, $PFH(X, \omega_X)_J$ is independent of the choice of such a class of  almost complex structures.

The idea of the proof is to establish a 1--1 correspondence between the holomorphic curves and solutions to Seiberg--Witten equations perturbed by the symplectic form.   It requires many aspects of the Seiberg--Witten equations and argument from Taubes's series of papers \cite{T1}, \cite{T2}, so it has beyond the scope of this paper.  We refer the reader to Section 8 of \cite{GHC} where proves a parallel result for other  Lefschetz fibrations. \cite{GHC} summarizes the main argument of Taubes and explains why the argument can be adapted to our setting.  The  only difference here with \cite{GHC}  is that there are extra covers of holomorphic planes, called special holomorphic planes,  contributed to the cobordism maps.  To deal with such curves, the argument has already been carried out by C.Gerig \cite{CG2}.  For a  class of symplectic cobordisms, Gerig proves that  the cobordism maps on embedded contact homology defined by holomorphic curves agree  with the cobordism maps on Seiberg--Witten Floer cohomology  via the isomorphism “ECH=SWF” in \cite{T2}.

Note that  Gerig's  setup  differs from ours in the following two aspects: Firstly, the triple $(Y, \pi, \omega)$ is replaced by a contact 3--manifold in \cite{CG2}. Also, the periodic orbits are changed to be Reeb orbits. Secondly, the Seiberg--Witten equations are perturbed by $\omega$ and $\Omega_X$ in \cite{GHC} while they are changed to be the contact form and the  symplectic form of a  symplectic cap respectively. These changes only influence the proof of the ``$SW \Rightarrow Gr$ '' degeneration because the other arguments mainly  take place locally near the holomorphic curves. The  ``$SW \Rightarrow Gr$ '' degeneration in our setting has been proved in Proposition 5.12 of \cite{GHC}. Therefore, the proof in \cite{CG2} can be applied to our setup with only notational changes.
\end{remark}

\begin{remark}
For the case $Q \le g(F)-1$,  the author conjectures that the   conclusion in the above remark should be still true.  The only difference in the current  case is that we need to use  certain tame almost complex structures to define PFH (see 1.1.4 of \cite{LT}) and its cobordism maps; otherwise, the fibers of $Y$   violate the compactness of the moduli space. However, the techniques in Taubes's papers   require a compatible almost complex structure rather than a tame one. To overcome this issue,  a solution is provided by Lee and Taubes \cite{LT}.  For each tame almost complex structure $J$ in the symplectization of $Y$, they modify the symplectic form such that $J$ is compatible with this new symplectic form.   If their construction can be generalized to the cobordism case, then we may run the same argument as in the case $Q>g(F) -1$ to define $HP_{sw}(X, \Omega_X)$ and show that $HP_{sw}(X, \Omega_X)=PFH(X, \omega_X)_J$.

We   believe that  $PFH(X, \omega_X)_J$ is independent of the almost complex structures defined it. If $2Q \le g(F)-1$, then  we already see this  from the computation in Theorem \ref{thm3}.
%The idea of the proof is to establish a 1-1 correspondence between the holomorphic curves and solutions to Seiberg Witten equations perturbed by symplectic form. In fact, \cite{GHC}  prove the parallel results for other Lefschetz fibrations. The only difference here with \cite{GHC}  is that there are extra covers of holomorphic planes  contributed to the cobordism maps.

%One could prove it by using the same argument as in  Gerig's paper.  a variant of the techniques in ``Gr=SW''
\end{remark}

\paragraph{Motivations}There are several motivations for this paper. First, while PFH is defined in terms of holomorphic curves,  the cobordism maps  currently are only defined by using the   Seiberg--Witten  theory.  %for  some special cases, for general case, we can define the map  via the isomorphism between PFH and  Seiberg Witten  cohomology \cite{LT}.
Even we can use the Seiberg--Witten theory as a `black box' to construct the cobordism maps on PFH, it is still meaningful to understand the cobordism maps in terms of holomorphic curves.   %We would like to put PFH in a purely symplectic setting.
This could help with finding higher--dimensional analogues of
this invariant. Also,  in many applications, the holomorphic curves definition seems more  suited for computation.
Secondly, in the previous result \cite{GHC}, the author   defines the cobordism maps by using the holomorphic curve methods for fiberwise symplectic cobordisms satisfying  certain technical assumptions $(\spadesuit)$. The elementary Lefschetz fibrations are the simplest non--trivial examples that  do  not satisfy the assumptions $(\spadesuit)$.  Therefore, the results here  help with generalizing the results in \cite{GHC} to more situations. Thirdly, the computation here should help with computing the cobordism maps $PFH(W, \omega_W)_J$, where $(W, \pi_W, \omega_W) $ is the  fiberwise symplectic cobordism  constructed in \cite{PS1}. It is a Lefschetz fibration over an annulus with only one critical point. %At least we know that $PFH(W, \omega_W)$  should be  nonvanishing by using the computation here and the composition law.

\paragraph{Idea of the proof} We end up the introduction  by  summarizing the   idea of the proof. First, we establish  a combinatorial formula for ECH index in Section \ref{section1}. From this index formula, we can see there are only a few generators that  can be mapped to non--zero.  Secondly,  there are special sections of $(X, \pi_X)$, called horizontal sections, whose positive ends are asymptotic to these potential generators. Thanks to  Seidel's observation,  these sections are holomorphic curves  for a suitable choice of  almost complex structures. Finally, the energy constraints ensure that the horizontal sections are the only holomorphic curves that  contributed to the cobordism maps. The above three ingredients lead to Theorem \ref{thm3}.

%We have organized the rest of this article in the following way:  In section $2$, we review PFH and the  definition which will be using.  In section $3$,  we review the construction of $\pi_X: X \to D$  and compute its singular homology.   In section $4$,  we describe  the periodic orbits on $Y$.    In section $5$,  a combinatorial formula for ECH index will be deduced. Finally, in section $6$, we show that the cobordism map is well defined and compute it.

\ack{
This work is carried out at the   University of Adelaide,  the author acknowledges the comfortable environment provided by the School of Mathematics.
He is grateful to Prof.Yi--Jen Lee for suggesting this project.
He also wants to thank the anonymous referee, whose comments and  suggestions have greatly improved this paper.
The author also thanks   the support from ARC Discovery project grant DP170101054 with Chief Investigators Mathai Varghese and David
Baraglia.}
\section{Preliminaries}

\subsection{Review of ECH index and periodic Floer homology }
In this section, we   review    those aspects  of  periodic Floer homology that we need in the proof of Theorem \ref{thm3}. For more details, please refer to \cite{H3}.

%\subsection{Basic definition}
%\begin{definition}
%Let $\pi_W: W \to B$ be a surface fibration over a circle or a  Lefschetz fibration. Given a $2$-form  $\omega_W \in \Omega^2(W)$, $\omega_W$ is called admissible if $d\omega_W=0$  and $\omega_W$ is fiberwise nondegenerate.
%\end{definition}
%Note that the  admissible $2$-form  $\omega_W$ gives a decomposition   $TW = TW^{hor} \oplus  TW^{vert}$ of the tangent bundle of $W$, where $TW^{vert}= \ker \pi_{W*}$ and $TW_{w}^{hor}=\{ a \in TW_w \vert   \omega_W(a, b)=0  \mbox{ } \forall  b \in  TW_w^{vert} \}$.

\paragraph{Orbit set} Let $(Y,  \pi)$ be a surface fibration over the  circle together with  a Hamiltonian structure $\omega$.
The $2$--form $\omega$ gives a splitting of $TY=TY^{vert}\oplus TY^{hor}$, where $TY^{vert}= \ker \pi_*$ and $TY^{hor}$ is the $\omega$--orthogonal complement  of $TY^{vert}$.
%The horizontal lift of $\partial_t$ is  called the Reeb vector field $R$, where  $\partial_t$  is coordinate vector field of $S^1$.  %A periodic orbit $\gamma$ of $(Y, \pi, \omega)$ is the periodic integral curve of $R$.
The \emph{Reeb vector field} $R$ of $(Y, \pi, \omega)$ is characterized by  conditions:
\begin{center}
$R \in \Gamma(TY^{hor})$  and  $ \pi_*(R) =\partial_t $,
\end{center}
where  $\partial_t$  is the coordinate vector field of $S^1$.
A periodic orbit is a smooth map $\gamma: \mathbb{R}_{\tau} / q \mathbb{Z} \to Y $ satisfying the ODE $\partial_{\tau} \gamma =R \circ \gamma$ for some $q>0$. The number  $q$ is called the \emph{period or degree} of the periodic orbit. We say that a periodic orbit $\gamma$ is elliptic if the eigenvalues of the linearized return map $P_{\gamma}$ are on the unit circle, positive hyperbolic if the eigenvalues of $P_{\gamma}$ are real positive numbers, and negative hyperbolic if the eigenvalues of $P_{\gamma}$ are real negative numbers.
%\begin{equation*}
%\gamma: \mathbb{R}/{d \mathbb{Z}} \to Y, \partial_t \gamma =R \circ \gamma(t),
%\end{equation*}
 %The Reeb vector field $R$ is a section of $TY^{hor}$ such that $\pi_*(R)=\partial_t$. A periodic orbit is a smooth map $\gamma: \mathbb{R}_{\tau} / q \mathbb{Z} \to Y $ satisfying the ODE $\partial_{\tau} \gamma =R \circ \gamma$ for some $q>0$.   The number $q$  is called period or degree of $\gamma$.

 An \emph{orbit set} $\alpha=\sum_i m_i \alpha_i$ is a finite formal sum of periodic orbits, where $\alpha_i $ are distinct, non--degenerate, irreducible embedded periodic orbits and $m_i$  are positive integers.
 An orbit set $\alpha$ is called an \emph{ECH generator} if $m_i=1$ whenever $\alpha_i$ is a hyperbolic orbit.  In the rest of the paper,  we write an orbit set using multiplicative notation $\alpha=\Pi_i \alpha_i^{m_i}$ instead of summation  notation.
% The intersection number $[\alpha] \cdot [fiber]$ is called the degree of $\alpha$.
The following definition is useful when we define the cobordism maps on PFH.

\begin{definition} (see \cite{H5} Definition 4.1)
Fix  $Q>0$.  Let  $\gamma$ be an embedded elliptic orbit with degree $q \le Q$.
\begin{itemize}
\item
$\gamma$ is called $Q$-positive elliptic if the rotation number $\theta  $ is in $ (0, \frac{q}{Q}) \mod 1$.
\item
$\gamma$ is called $Q$-negative elliptic if the rotation number $\theta $ is in $ ( -\frac{q}{Q},0) \mod 1$.
\end{itemize}
\end{definition}

\paragraph{The ECH index} Let  $(W , \pi_W, \omega_W)$ be a fiberwise symplectic  cobordism from  $(Y,  \pi, \omega)$ to $(Y',  \pi', \omega')$, where $(W, \pi_W)$ is a Lefschetz  fibration such that  $\pi_W^{-1}(\partial B)= Y \cup(-Y')$,  and $\omega_W $ is  a fiberwise symplectic  form which agrees with $\omega$ and $\omega'$ along $Y$ and $Y'$ respectively.

Given orbit sets $\alpha=\Pi_i  \alpha_i^{m_i}$ and $\beta=\Pi_j  \beta_j^{n_j}$ on $Y$ and $Y'$ respectively,  define the space of relative homology classes $H_2(W, \alpha ,\beta)$.  A  typical element in $H_2(W, \alpha ,\beta)$ is a $2$--chain $Z$ in $W $ such that $\partial Z = \sum_i m_i \alpha_i- \sum_j n_j \beta_j$, modulo the boundary of $3$--chains. Note that  $H_2(W, \alpha ,\beta)$ is an affine space over $H_2(W, \mathbb{Z})$.
Given $Z \in H_2(W, \alpha ,\beta)$ and  trivializations    $\tau$ of $\ker \pi_* \vert_{\alpha}$ and $\ker \pi_*' \vert_{\beta}$,  the ECH index is defined by
\begin{equation*}
I(\alpha, \beta, Z) = c_{\tau}(Z) + Q_{\tau}(Z) + \sum_i  \sum\limits_{p=1}^{m_i} CZ_{\tau}(\alpha_i^{p})- \sum_j \sum\limits_{q=1}^{n_j} CZ_{\tau}(\beta_j^{q}),
\end{equation*}
where $c_{\tau}(Z)$ and $Q_{\tau}(Z)$ are respectively the relative Chern number and the relative self--intersection number  (see \cite{H4} and \cite{H2}), and $CZ_{\tau}$ is the Conley--Zehnder index.  The  ECH index $I$  depends only on orbit sets $\alpha$, $\beta$, and the  relative homology class  $Z$.

%\subsection{Almost complex structure}

%The symplectic form on $W$ is given by $\Omega_W=\omega_W+\pi^*\omega_B$, where $\omega_B$ is a large volume form of $B$. %Set $\omega=\omega_X \vert_{Y}$.  Assume that there is a neighborhood $U=[-\epsilon, 0]\times Y$ of $Y$ such that $\omega_X \vert_U=\omega$, then
% We can define the symplectic completion $(\overline{W}, \omega_W)$ by adding a cylindrical end. (See section 2.3 of \cite{GHC}).

\paragraph{Almost complex structure}  Recall that the symplectic form on $W$ is defined  by $\Omega_W=\omega_W+\pi^*\omega_B$, where $\omega_B$ is a large volume form of $B$. %Set $\omega=\omega_X \vert_{Y}$.  Assume that there is a neighborhood $U=[-\epsilon, 0]\times Y$ of $Y$ such that $\omega_X \vert_U=\omega$, then
 We can define the \emph{symplectic completion} $(\overline{W}, \omega_W)$ by adding  cylindrical ends. Moreover, the fibration structure can be extended to the completion. More details can be found in  Section 2.3 of \cite{GHC}.
 To introduce holomorphic curves in $ \overline{W} $, we need the following definition of almost complex structures.
\begin{definition}
An almost complex structure $J$ on $\overline{W}$  is called $\Omega_W$--tame if $J$ satisfies the following properties:
\begin{enumerate}
\item
On the cylindrical ends, $J$ is $\mathbb{R}_s$--invariant and $J(\partial_s)=R$, where $R$ is the Reeb vector field. Also, we require that   $J $ sends $\ker \pi_*$ to itself along periodic orbits with degree less than  or equal to $Q$
\item
$J$ is $\Omega_W$--tame.
\item
Identify a neighbourhood  of a  critical point   of $\pi_W$ with the local model (\ref{eq18}). Then $J$ agrees with the standard complex structure $J_0$ in that neighbourhood.
\end{enumerate}
\end{definition}
The space of $\Omega_W$--tame  almost complex structures is denoted by $\mathcal{J}_{tame}(W, \omega_W)$. It carries a natural $C^{\infty}$ topology.

For the purpose of computing the cobordism maps, we require  a special class of almost complex structures that  preserve  the vertical and horizontal bundles.  Recall that the admissible 2--form  $\omega_W$ splits $TW$ into $TW^{hor}$ and $TW^{vert}$, where  $TW^{vert}=\ker \pi_{W*}$ and $TW^{hor}=\{v\in TW \vert v \bot_{\omega_W} TW^{vert} \}$.
Let $\overline{B}$ be the completion of $B$ by adding cylindrical ends. Let $(s, t)$ be the coordinates of the cylindrical ends. We fix a $\omega_B$--compatible complex structure $j_B$ such that $j_B(\partial_s) =\partial_t$ on the cylindrical ends.

\begin{definition}\label{def2}
%Fix a complex structure $j_B$ over the base surface $B$. 
Let $\mathcal{J}_h(W, \omega_W) \subset \mathcal{J}_{tame}(W, \omega_W)$ be  the subspace of $\Omega_W$--tame almost complex structures on $\overline{W}$ with the following properties:
\begin{enumerate}
\item
Away from critical points of $\pi_W$, $J(T\overline{W}^{hor}) \subset T\overline{W}^{hor}$ and $J(T\overline{W}^{vert}) \subset T\overline{W}^{vert}$.
\item
Away from critical points of $\pi_W$, $J \vert_{\ker \pi_{W*}}$ is compatible with $\omega_W$.
\item
$(\pi_W)_*$ is complex linear with respect to $J$ and $j_B$, i.e., $j_B \circ (\pi_W)_*=(\pi_W)_* \circ J$.
\end{enumerate}
\end{definition}
\begin{remark}
For a generic $J \in \mathcal{J}_h(W, \omega_W)$, all simple  holomorphic curves can achieve transversality except for the fibers and horizontal sections. See Proposition 21.1 of \cite{PS1}.  For the definition of horizontal sections, please refer to  Definition \ref{def1}.
\end{remark}
\textbf{Notation.} The definition of   $\mathcal{J}_h(W, \omega_W)$ depends  obviously  on the choice of $j_B$, but it plays no role in our argument.  To  simplify the notation, we suppress  $j_B$  from the notation.

\paragraph{Holomorphic currents}
Let $(\overline{W} , \pi_W, \omega_W)$ be  the symplectic completion of $(W , \pi_W, \omega_W)$.  Fix an   $\Omega_W$--tame almost complex structure $J$. A $J$--\emph{holomorphic current} $\mathcal{C}= \sum d_a C_a$  from $\alpha$ to $\beta$ is a finite formal sum of  $J$--holomorphic curves such that  $\mathcal{C}$ is asymptotic to $\alpha$ and $\beta$ respectively in the current sense, where  $C_a$  are distinct, irreducible, somewhere injective $J$--holomorphic curves with finite energy $\int_{C_a} \omega_W < \infty$ and $ d_a$  are positive integers.
The number
\begin{equation*}
 E(\mathcal{C})=\int_{\mathcal{C}} \omega_W =\sum_a d_a \int_{C_a} \omega_W
\end{equation*}
is called  $\omega_W$--\emph{energy}.  Let  $\mathcal{M}_i^J(\alpha, \beta, Z)$ denote the space of holomorphic currents from $\alpha$ to $\beta$ with ECH index $I=i$ and relative homology class $Z$.

\begin{definition}
 A holomorphic current $\mathcal{C}=\sum_a d_a C_a$ is called embedded if $d_a=1$ for any $a$ and $C_a$  are pairwise disjoint embedded holomorphic curves.
\end{definition}

\paragraph{Definition of periodic Floer homology }
Now let us return to the definition of PFH.  Fix an integer $Q$.  The chain complex $PFC(Y, \omega, Q)$ of PFH  is a free module  generated by  ECH generators with degree $Q$. Consider the special case that $\overline{W}=\mathbb{R} \times Y$ and $J$  is a generic $\mathbb{R}$--invariant $\Omega_W$--tame almost complex structure. The differential of PFH  is defined by
\begin{equation} \label{eq19}
<\partial \alpha, \beta> =  \sum_{Z \in H_{2}(Y, \alpha, \beta)} \left(\# {\mathcal{M}}^J_{1}(\alpha, \beta, Z)/\mathbb{R} \right)
\end{equation}
%The condition $I=1$ implies that the holomorphic currents are embedded possibly with unbranched covers of trivial cylinders (see \cite{H4}). Thus the moduli space is a zero dimensional manifold.    The compactness in \cite{H4} ensures that the  counting  of ${\mathcal{M}}^J_{1}(\alpha, \beta, Z)/\mathbb{R}$ makes sense.
  The obstruction gluing  argument  in \cite{HT1} and \cite{HT2} show that $\partial^2=0$.   The homology of $(PFC(Y, \omega, Q), \partial)$ is called \emph{periodic Floer homology}, denoted by $PFH_*(Y, \omega, Q)$.
% We use $PFH_*(Y, \omega, Q)=\oplus_{\Gamma \cdot [fiber] =Q} PFH_*(Y, \omega, \Gamma)$.

\paragraph{Cobordism maps on PFH}
Let  $(W , \pi_W, \omega_W)$ be a fiberwise symplectic cobordism from  $(Y,  \pi, \omega)$ to $(Y',  \pi', \omega')$. It is expected to define the cobordism maps in chain level by
\begin{equation} \label{eq2}
<PFC(W, \omega_W)_J \alpha, \beta  >= \sum_{Z \in H_{2}(W, \alpha, \beta)} \left(\# {\mathcal{M}}^J_{0}(\alpha, \beta, Z) \right) .
\end{equation}
However, the above formula doesn't make sense in general due to the  appearance of holomorphic currents with negative ECH index. The reasons are explained in Section 5.5 of \cite{H3}.
But in our special case, we show that it actually  works.

\subsection{  Elementary Lefschetz fibration} \label{section3}
%\subsection{ Basic definition}
%Let  $\pi_X: X \to D$ be a Lefschetz fibration over a disk with only one singular point at origin. The boundary of $X$ is denoted by $Y$. Throughout this note, we assume that the fiber $F$ of $X$ with genus $g(F) \ge 2$, also, the vanishing cycle is non-separated.

%\subsection{construction of Elementary  Lefschetz fibration}
In this subsection, we give more details about the elementary  Lefschetz fibration $(X, \pi_X, \omega_X)$.  To this end, let us first return to the exact Lefschetz fibration $(E, \pi_E, \omega_{\mathbb{C}^2} \vert_E)$. What follows here paraphrase  parts of the accounts in  \cite{PS1}, \cite{PS}.%Let $\pi: \mathbb{C}_{\textbf{x}=(x_1, x_2)}^2 \to \mathbb{C}$ be the local model of Lefschetz fibration and $\Sigma$ be the union of all vanishing cycle (include $(0,0)$), where $\pi(\textbf{x})=x_1^2 +x_2^2$.

Let $T=T^*S^1$ be the cotangent bundle of  the circle. The standard coordinates on  $T$ are
 $$T=\{ (u,v) \in \mathbb{R}^2 \times \mathbb{R}^2: <u, v>=0  \ \ |v|=1\}. $$
In terms of the coordinates,  the geodesic flow  $\sigma_t$  on  $T$ is
\begin{equation*}
\sigma_t(u ,v)= \left(\cos(2\pi t) u -\sin (2\pi t) |u| v,  \cos (2\pi t) v + \sin(2\pi t) \frac{u}{|u|}\right).
\end{equation*}
Sometimes it is more convenient to identify   $T$ with  a cylinder $\mathbb{R}_x \times  (\mathbb{R}_y/\mathbb{Z})$  via changing of coordinates
\begin{equation}\label{eq7}
u=ixe^{i2\pi y}, \ \ v=e^{i2\pi y}.
\end{equation}
Let $T_{\lambda}=\{(u, v) \in T \vert |u| \le \lambda\}$.  Let $V=\cup_{z \in D_{\delta}} V_z \cup {(0,0)} $ be the union of all the vanishing cycles (include $(0,0)$) of $E$, where $V_z =\sqrt{z}\{ (0, v) \in T \}\subset \pi_E^{-1}(z)$.   Away from $V$, $E$ can be identified with  a trivial bundle over a disk via the following diffeomorphism:
\begin{equation} \label{eq9}
\begin{split}
&\Phi: E-V \to D_{\delta}\times (T_{\lambda}-T_0)\\
&\Phi(\textbf{x})=(\pi(\textbf{x}),  \sigma_{\frac{t}{2}} (-Im(\hat{\textbf{x}}) |Re(\hat{\textbf{x}})|,  Re(\hat{\textbf{x}}) |Re(\hat{\textbf{x}})|^{-1} ) )=(\pi(\textbf{x}), \sigma_{\frac{t}{2}}(\Phi_2(\hat{\textbf{x}}))),
\end{split}
\end{equation}
where $\Phi_2(\textbf{x})=(-Im(\hat{\textbf{x}}) |Re(\hat{\textbf{x}})|,  Re(\hat{\textbf{x}}) |Re(\hat{\textbf{x}})|^{-1} ) )$   and  $\hat{\textbf{x}}= e^{-i\pi {t} }\textbf{x} $ provided that $\pi(\textbf{x}) =re^{2\pi it}$. $\Phi$ satisfies the following properties:
\begin{enumerate}
\item
$\Phi$ is a diffeomorphism on each fiber.
\item
$(\Phi^{-1})^* \theta_{\mathbb{C}^2}=\theta_T - \tilde{R}_r(|u|) dt,$ where $\theta_{\mathbb{C}^2 } =\frac{i}{4}\sum_{k=1}^2(z_k d\bar{z}_k - \bar{z}_k dz_k)$,   $\tilde{R}_r(t)=  \frac{t}{2} - \frac{1}{4} \sqrt{r^2 + 4 t^2}$ and  $\theta_T = xdy$ in terms of coordinates (\ref{eq7}).
\end{enumerate}
%Now let us fix a $\lambda >0$ and $\delta>0$.   Define
%$$E=\Phi^{-1}(D_{\delta} \times (T_{\lambda}-T_0) )\cup (\Sigma \cap \pi^{-1}(D_{\delta})).$$
%Alternative, $E=\{ \textbf{x} \in \mathbb{C}^2 \vert  |\textbf{x}|^4-|\pi(\textbf{x})|^2 \le 4 \lambda^2, \ \ |\pi(\textbf{x})| \le \delta\}$. %$E$ is a manifold with corner,  the boundary of $E$ can be divided into vertical boundary $\partial_v E=\pi^{-1}(\partial D_{\delta})$ and horizontal boundary $\partial_h E = \Phi^{-1}(D_{\delta} \times \partial T_{\lambda})$.

As mentioned in the Introduction, we need to modify $\omega_{\mathbb{C}^2}$ so that it is ``trivial''  near the horizontal boundary in the sense that it agrees with $d\theta_T$.      Fix $ 0< \delta_0 \ll 1$.    Take a cutoff function $g$ such that $g(t)=0$ near $t=0$ and $g(t)=1$ where $t \ge \lambda-\delta_0$. %Define a 1-form $\mathfrak{r}=\Phi^*(g(|x|) \tilde{R}_r(|x|) dt )$.
Define a 1--form by
 $$\theta_E= \Phi^*( \theta_T + (g(|x| ) -1) \tilde{R}_r(|x|) dt ).$$
% by cutoff $\tilde{R}_r(|x|)$ near the boundary $\partial T_{\lambda}$.   
Then the 2--form $\omega_E=d\theta_E$ satisfies the requirement.  

 Let $N(\partial_h E )= D_{\delta} \times ([-\lambda, -\lambda + \delta_0] \cup [\lambda- \delta_0, \lambda]) \times S^1 $ be a neighborhood of the horizontal boundary.   Let $(S, \omega_S)$ be a connected symplectic surface with boundary $\partial S =S^1 \cup (-S^1)$. A collar neighborhood of $\partial S $ is identified with $N(\partial S ) =  ([-\lambda, -\lambda + \delta_0] \cup [\lambda- \delta_0, \lambda]) \times S^1$.  Choose $\omega_S$ such that   $\omega_S \vert_{N(\partial S)} = d \theta_T$.  Take a trivial fiberwise  symplectic fibration over a disk $(D_{\delta} \times S, \omega_S)$. Then  $\pi_X: X \to D_{\delta}$ is obtained by gluing $E$ with $D_{\delta} \times S$  via identifying   $N(\partial_h E)$ and   $D_{\delta} \times N(\partial S)$.
%Choose a suitable $\omega_S$ such that it agrees with $\omega_E$ along $E \cap (D \times S, \omega_S) $.
 The admissible 2--form $\omega_X$ is obtained by  gluing   $\omega_E$ with $\omega_S$ in the obvious way. The fiber of $X$ is $F=S \cup T_{\lambda}$ and $Y=\partial_v E\cup( S \times S^1)$.
%Let $\theta_E=\theta_{\mathbb{C}^2} + \mathfrak{r}$ and  $\omega_E= d\theta_E$ be  the admissible $2$-form on $E$.  By definition, $\theta_E= \Phi^*( \theta_T + (g(|x| ) -1) \tilde{R}_r(|x|) dt )$. Let $R_r(t) = (1-g(t)) \tilde{R}_r(t)$, then $\theta _E=\Phi^*( \theta_T -R_r(|x|) dt) $.

Lemma 18.4 of \cite{PS1} shows that the symplectic monodromy   $\phi: \pi^{-1}(\delta) \to  \pi^{-1}(\delta)  $  of $(E, \omega_E)$    is given by
\begin{equation*}
\phi(u, v)  = \left\{
\begin{aligned}
& \sigma_{ R_{\delta}'(|u|)} (u, v)  \   \mbox{ if $ u \ne 0$} \\
&  (0, -v)\  \mbox{ if $u=0$.}
\end{aligned}
\right.
\end{equation*}
Here $R_r(t)=(1-g(t))  \tilde{R}_r(t)$.
The above monodromy  is called the   \emph{Dehn twist}.  We  extend $\phi$  to be $Id$ over $S$.

Although $\phi$ depends on   the choices of $\delta, \lambda$, and the cutoff function $g$,  it is unique up to a Hamiltonian isotopy. (See Lemma 2.1 of \cite{PS1}.)
Since the periodic Floer homology is invariant under a    Hamiltonian isotopy of the monodromy (see Corollary 1.1 of \cite{LT}), we don't emphasize these choices unless otherwise stated.
%$\phi(x, y) =\sigma_{ R_{\delta}'(|x|)} (x, y)$ when $x \ne 0$ and $\phi(0, y)=y\pm \frac{1}{2}$. As a consequence, the vertical boundary $\partial_v E$ is a mapping torus of Dehn twisted.
%According to the construction,  for any $J \in \mathcal{J}_h(X, \omega_X)$, $(\pi_X: X \to D, J)$ is nonnegative.
In the rest of the paper, we always assume that $\delta=1$ and denote $R(t)=R_1(t)$. Also, we choose a suitable cutoff function $g$ such that $0 \le  R'(|t|)\le \frac{1}{2}$ and $R''(|t|) \le 0$.

The following lemma concerns the condition $\int_F \omega_X  \ge Q +1$ in Theorem \ref{thm3}. We can always construct an admissible 2--form satisfying this condition.
\begin{lemma} \label{lem1}
Given $L>0$, we can   find an admissible $2$--form $\omega_X$ such that $\int_F \omega_X>L$.
\end{lemma}
\begin{proof}
Let $\omega_X$ be the admissible $2$--form  constructed before.  We   modify $\omega_X$ such that $\int_F \omega_X >L$    as follows.

%  A  neighborhood $N$ of     $\partial S $ can be identified with $N=( [0, 1+ \epsilon] \times S_y^1 ) \cup ( [ -1-  \epsilon, 0] \times S_y^1 )  $.   The volume form $\omega_S = d\theta_S = \omega_E$  on $( [1, 1+ \epsilon] \times S^1 ) \cup ( [ -1-  \epsilon, -1] \times S^1 )  $.
Let $U= ([-\lambda-2\delta_0, -\lambda + \delta_0] \cup [\lambda- \delta_0, \lambda + 2\delta_0]) \times S_y^1$ be a collar neighborhood of     $\partial S $.
Take a function $f: S \to \mathbb{R}$ supported in $U$ such that  $f= c_1 $ on $ [\lambda-\delta_0, \lambda+ \delta_0] \times S^1 $ and $f= c_2 $ on  $( [ -\lambda-\delta_0, - \lambda +\delta_0] \times S^1 ) $,  where $c_1$ and $c_2$ are positive constants.  Write $\omega_S=d\theta_S$ for some 1--form $
\theta_S$. Define $\theta_S' = \theta_S + f dy$. Then we still have  $\omega_S'=d\theta_S'=\omega_E$ on $N(\partial S) $ by definition.   Thus we can glue $\omega_S'$  with $\omega_E$ together as before. The result is called $\omega'_X$.  By  Stokes’ theorem, we have
 \begin{equation*}
 \int_F \omega_X' =  \int_F \omega_X + c_1 -c_2.
 \end{equation*}
%Therefore, given a constant $c_0$, we can always choose suitable function $f$ such that $\int_F \omega_X' > c_0$.
Take $c_1-c_2>L$.  Then we get the conclusion.
\end{proof}

%\subsection{Singular homology of $Y$ and $X$}
%In this subsection, let us compute the classical invariant of $X$ and $Y$, where $Y=\partial X$ is mapping torus of the Dehn twisted $\phi$, that is
%\begin{eqnarray*}
%Y= \mathbb{R}\times F /(t+ 2\pi, x) \sim (t, \phi(x)).
%\end{eqnarray*}
%
%\begin{lemma}
%$H_1(Y , \mathbb{Z})=\mathbb{Z}^{2g(F)}$.
%\end{lemma}
%\begin{proof}
%According to the elementary  algebraic topology, we have
%% \begin{equation*}
%%\begin{split}
%%H_1(F, \mathbb{Z})  \xrightarrow{ 1-\phi_*}H_1(F, \mathbb{Z}) \to  H_1(Y, \mathbb{Z}) \to H_0(F , \mathbb{Z})  \to 0.
%%\end{split}
%% \end{equation*}
%$H_1(Y , \mathbb{Z})=  H_0(F , \mathbb{Z}) \oplus Coker\{(1-\phi_*)\vert_{H_1(F, \mathbb{Z})}\}$.  Says $\phi$ is a Dehn twisted around curve $a$, then $\phi_*(b)=b+ (a\cdot b)  a$ for any $b \in H_1(Y, \mathbb{Z})$. Using this relation, it is easy to check that $ Coker\{(1-\phi_*)\vert_{H_1(F, \mathbb{Z}) }\}=\mathbb{Z}^{2g(F)-1}$.
%\end{proof}
%Let us recall that the construction of $X$ in \cite{PS1}.  Let $E=\{(x, z ) \in \mathbb{C}^2 \times D\vert |x|^2 < 2 \lambda, \ \ x_1^2 + x_2^2 =\xi(|x|^2) z\}$, where $\xi$ is a function such that $\xi=1$ on $t< \frac{1}{2}\lambda$ and $\xi=0 $ on $t \ge \lambda$.  It is easy to check that $U=\{(x, z ) \in E \vert |x|< \lambda \}$  is homeomorphic to $\mathbb{C}^2$ and $E^{trivi}=\{(x, z) \in E \vert |x| \ge \lambda\}$ is a trivial bundle over $D$ with fiber $I \times S^1$.  $X$ is obtained by gluing $E$ with a trivial fiber bundle $D \times S$ along $E^{trivi}$.

\begin{lemma}  \label{lem2}
We have $H_2(X, \mathbb{Z})=\mathbb{Z}$. Moreover, the generator is represented by the fiber. %, $H_2(X, \partial X, \mathbb{Z}) =\mathbb{Z}$ and $H_1(X, \mathbb{Z})=\mathbb{Z}^{2g(F)-1}$.
\end{lemma}
\begin{proof}
Recall that $X$ is obtained by gluing $E$ and $D \times S$ along their horizontal boundaries. The conclusion follows from the  Mayer--Vietoris theorem for the singular homology.
\end{proof}

\subsection{Periodic orbits} \label{section2}
In this section, we describe the periodic orbits on $(Y, \pi, \omega)$.
%Now let us assume that $\delta=1$. Also, we choose a suitable cut off function $g$ such that $0 \le  R'(t)<\frac{1}{2}$ and $R''(t) <0$.
Under the diffeomorphism (\ref{eq9}) and the coordinates (\ref{eq7}), the   Dehn twist is
\begin{equation*}
\phi(x, y)  = \left\{
\begin{aligned}
& (x, y+  R'(x))  \   \mbox{ if $ x > 0$} \\
& (x, y-  R'(-x))   \   \mbox{ if $ x < 0$} \\
&  (x, y \pm \frac{1}{2})\  \mbox{ if $x=0$.}
\end{aligned}
\right.
\end{equation*}
%Under the diffeomorphism \ref{eq9}, the admissible $2$-form over $S^1_t \times (T_{\lambda} -T_0)$ is given by
%\begin{equation*}
%\omega_E = \left\{
%\begin{aligned}
%& dx \wedge dy -R'(x) dx \wedge dt  \   \mbox{ when $x>0$} \\
%& dx \wedge dy +R'(-x) dx \wedge dt \  \mbox{ when $x<0$}
%\end{aligned}
%\right.
%\end{equation*}
%Therefore,  the Reeb vector field is
%\begin{equation*}
%R = \left\{
%\begin{aligned}
%& \partial_t  +R'(x) \partial_y  \   \mbox{ when $x>0$} \\
%&   \partial_t  -R'(-x) \partial_y  \  \mbox{ when $x<0$}
%\end{aligned}
%\right.
%\end{equation*}

%$$R=\partial_t  +R'(x) \partial_y$$ when $x>0$ and $$R= \partial_t  -R'(-x) \partial_y$$ when $x<0$.

At each $x_0$ such that $R'(x_0) = \frac{p}{q}$,  the  torus $T_{x_0}=S_t \times\{x_0\} \times S_y$ is foliated by embedded periodic orbits. Each  periodic orbit is of the form $\gamma_{\frac{p}{q}}(\tau)=(\tau, x_0, y_0+ R'(x_0) \tau)=(\tau, x_0, y_0+ \frac{p}{q} \tau)$ if $x_0>0$ and $\gamma_{\frac{p}{q}}(\tau)=(\tau, x_0, y_0- R'(-x_0) \tau)= (\tau, x_0, y_0 +(\frac{p}{q}-1)\tau)$ if $x_0<0$ (When $x_0<0$, we write $R'(-x_0)=1-\frac{p}{q} $.), where $\tau \in  \mathbb{R}/( q\mathbb{Z})$. The integers $p, q$ here are relatively prime. The torus $T_{x_0}$ is called a  \emph{Morse--Bott torus}.

To ensure that the periodic orbits are non--degenerate, we need to perturb the admissible 2--form $\omega$.
    At each Morse--Bott torus $T_{x_0}$, we can  perform a standard  perturbation  such that there are only two periodic orbits   that  survive  (see \cite{FB}). One is elliptic and the other one is positive hyperbolic, denoted by $e_{\frac{p}{q}}$ and $h_{\frac{p}{q}}$ respectively.  These two periodic orbits are corresponding to the minimum and maximum of a perfect Morse function $f_{T_{x_0}}$ on the circle.   For  any fixed  integer $Q$,  we can arrange that all the periodic orbits on $\partial_v E$ with degree less than or equal to $Q$ are either $e_{\frac{p}{q}}$ or $h_{\frac{p}{q}}$.

 To describe  the periodic orbits in the trivial part $S^1 \times S$ of $Y$,  let $f_S: S \to \mathbb{R}$  be a   small Morse function such that $\nabla f_S $ is transversal to $\partial S$ and there are  critical points  $\{p_i\}_{i=0, 1}$ and  $\{q_i\}_{i=0, 1}$ on $N(\partial S) $ satisfying $\nabla^2 f_S(p_i)>0$ and $tr(\nabla^2 f_S(q_i)) =0$. In addition, there are no other critical points on $N(\partial S)$.
 For sufficiently small $f_S$,
 %We use $f$ to perturb $\omega_S$ by $\omega_S^f =\omega_S + df \wedge dt $, so that
 the periodic orbits with degree less than or equal to $Q$ over $(S^1 \times S, \omega_S + df_S \wedge dt)$ only consists of constant orbits at critical points of $f_S$.  We always use $a \in Crit(f_S)$ to denote the critical point  on $S - N(\partial S )$. If $tr(\nabla^2f_S (a)) \ne 0$,  then the corresponding periodic orbit is elliptic, denoted by $e_a$. Moreover, $e_a$ is either $Q$--positive or $Q$--negative depending on the sign of $\nabla^2 f_S(a)$ accordingly.  If $tr(\nabla^2f_S (a))=0$, then  the corresponding periodic orbit is positive hyperbolic, denoted by $h_a$. $\{e_i\}_{i=0, 1}$ and $\{h_i\}_{i=0, 1}$ are respectively elliptic orbits and hyperbolic orbits corresponding to $\{p_i\}_{i=0, 1}$ and  $\{q_i\}_{i=0, 1}$.

 In conclusion,  we can arrange that all the periodic orbits with degree less than or equal to $Q$ are either $e_{\frac{p}{q}}$ or $h_{\frac{p}{q}}$ or $e_a$ or $h_a$. Keep in mind that the periodic orbits here are either $Q$--positive elliptic or $Q$--negative elliptic or positive hyperbolic.

\begin{remark}
 In order to ensure that the elementary Lefschetz fibration is nonnegative in the  sense of Definition \ref{def1},  the   functions $f_{T_{x_0}}$ and $f_S: S\to \mathbb{R}$ are chosen to be non--negative and their minimums are zero.

% For the purpose of computation, we require that $e_0$ and $e_1$ are the only  critical points with $\nabla^2 f>0$.
 \end{remark}
 %We denote the orbit by $\gamma_{\frac{p}{q}}$.
% With respect to the trivialization induced by \ref{eq9},  we can use small Morse function to perturb the admissible $2$-form so that there is only two periodic orbits at  Morse-Bott torus  survives, one is elliptic and the other one is positive hyperbolic, denoted by $e_{\frac{p}{q}}$ and $h_{\frac{p}{q}}$. The elliptic orbits has small negative rotation number, thus
%\begin{equation}\label{eq6}
%CZ_{\tau}(e_{\frac{p}{q}}) = -1,  \ \ \ CZ_{\tau}(h_{\frac{p}{q}}) = 0.
%\end{equation}

\begin{remark}
The description here  in fact is the same as  \cite{H1}. The model of the Dehn twist using by Hutchings and  Sullivan is slightly different from ours here.  But we can transfer our model to theirs via changing of coordinates and  a Hamiltonian isotropy.

Define  $f: [-\lambda, \lambda ]_x \times S_y^1  \to [0,1]_s \times S^1_t$ by sending $(x, y) $ to $(\frac{x}{2\lambda} + \frac{1}{2}, y)$.  Let $H(x)=R(x)+ \frac{x^2}{4\lambda} - \frac{x}{2}$ be a Hamiltonian function. Let $\phi_H$ be the time--1 flow of the Hamiltonian vector field of $H$. Then
\begin{equation*}
f\circ \phi_H \circ \phi \circ f^{-1}(s, t) = \left\{
\begin{aligned}
& (s, t-s+1)  \   \mbox{ if $\frac{1}{2} \le s \le 1$} \\
&  (s, t-s ) \  \mbox{ if $0\le s <\frac{1}{2}$}
\end{aligned}
\right.=(s, t-s).
\end{equation*}
coincides with the model in \cite{H1}.
\end{remark}
%Take $J \in \mathcal{J}_{h}(E, \omega_E)$,  near $Y$,  $T_J ^{1,0}E=< ds+idt, dx+i dy + \frac{1}{2}R'(x)(ds-idt)>$ when $x>0$ and $T_J ^{1,0}E=< ds+ idt, dx+i dy - \frac{1}{2}R'(-x)(ds-idt)>$ when $x<0$.
%
% Let $T=\{(u,  v) \vert  <u ,v>=0,  |v|=1\}$. We can identify $T$ with $\mathbb{R}_x \times S^1_y$ by
%\begin{equation} \label{eq7}
%v=e^{iy} \  and  \  u=xie^{iy}.
%\end{equation}
%Given a periodic  orbit $\gamma_{\frac{p}{q}}$, we want to find the inverse image $\textbf{x}(\tau)=\Phi^{-1}(\gamma_{\frac{p}{q}})$.

For the purpose of computing the ECH index, we   want to express the periodic orbits in terms of coordinates $ \textbf{x} =(x_1, x_2 )$.  The result is summarized  in the following lemma.
\begin{lemma}
Let $x_0 \in [-\lambda, \lambda]$ such that $R'(x_0)=\frac{p}{q}$ and $\gamma_{\frac{p}{q}}$ is the periodic orbit at $x_0$. Let $\textbf{x}(\tau)=(x_1(\tau), x_2(\tau))=\Phi^{-1}(\gamma_{\frac{p}{q}})$, $\tau \in \mathbb{R} /(q \mathbb{Z})$.  We have the following two cases:
If $x_0 \ge 0$,  then
%\begin{equation} \label{eq16}
%\begin{split}
%&x_1= \frac{1}{2} e^h e^{i (y_0+ \frac{p}{q} \tau)} +  \frac{1}{2} e^{-h} e^{-i(y_0 +\frac{p}{q} \tau -\tau )}\\
%&x_2=- \frac{i}{2} e^{h} e^{i(y_0 + \frac{p}{q} \tau)}+ \frac{i}{2} e^{-h} e^{-i (y_0 + \frac{p}{q}\tau- \tau)} ,
%\end{split}
%\end{equation}
\begin{equation} \label{eq16}
\left\{\begin{split}
&x_1= \frac{1}{2} e^{h+iy_0} e^{2\pi i  \frac{p}{q} \tau} +  \frac{1}{2} e^{-h-iy_0} e^{2\pi i \frac{q-p}{q} \tau  }\\
&x_2=- \frac{i}{2} e^{h+iy_0} e^{2\pi i \frac{p}{q} \tau}+ \frac{i}{2} e^{-h-iy_0} e^{2\pi i  \frac{q-p}{q}\tau} ,
\end{split} \right.
\end{equation}
for some constant $h=h(x_0) \ge 0$. If  $x_0 \le 0$, then
\begin{equation} \label{eq17}
\left\{\begin{split}
&x_1= \frac{1}{2} e^{h-iy_0} e^{ 2\pi i \frac{q-p}{q} \tau } +  \frac{1}{2} e^{-h+iy_0} e^{2 \pi i \frac{p}{q} \tau }\\
&x_2=\frac{i}{2} e^{h-iy_0} e^{ 2\pi i  \frac{q-p}{q}\tau} - \frac{i}{2} e^{-h +iy_0} e^{2\pi i \frac{p}{q} \tau},
\end{split} \right.
\end{equation}
for some constant $h=h(x_0)\ge0$. In  both cases, $h=0$ if and only if $x_0=0$.
\end{lemma}

\begin{proof}

Firstly,  one can check that the geodesic flow $\sigma_t$   can be written as  %$\sigma_t(x, y)=(x, y+ t)$ if $x>0$ and $\sigma_t(x, y)=(x, y-t)$ if $x<0$.
\begin{equation*}
\sigma_t(x, y) = \left\{
\begin{aligned}
& (x, y \pm t)  \   \mbox{ if $\pm x>0$} \\
&  (0, y \pm \frac{1}{2}) \  \mbox{ if $x=0$}.
\end{aligned}
\right.
\end{equation*}
 under the identification (\ref{eq7}).
Let $\textbf{x} \in E$ and $\hat{\textbf{x}}=e^{-i\pi {t}} \textbf{x} =\hat{p} + i\hat{q}$.  Write $\Phi_2(\hat{\textbf{x}})=(u, v) \in T$   as  $v= \frac{\hat{p}}{|\hat{p}|} = e^{i2\pi \theta}$ and $u=-\hat{q} |\hat{p}|=\pm |\hat{p}| |\hat{q}| i e^{i 2\pi \theta}$.

 %If $x= |\hat{p}| |\hat{q}| >0 $, then $v= \frac{\hat{p}}{|\hat{p}|} = e^{i\theta}$ and   $u=-\hat{q} |\hat{p}|= |\hat{p}| |\hat{q}| i e^{i\theta}$.
If $ \gamma_{\frac{p}{q}}(\tau)=(\tau, x_0, y_0+ \frac{p}{q} \tau)$ and $x_0>0$,  then
\begin{equation*}
\sigma_{\frac{t}{2}}(\Phi_2(\hat{\textbf{x}})) = \sigma_{\frac{t}{2}}(  -\hat{q} |\hat{p}|,   \frac{\hat{p}}{|\hat{p}|}  ) =(|\hat{p}| |\hat{q}|,    \theta + \frac{t}{2}   ).
\end{equation*}
 under the identification (\ref{eq7}).
%At $x_0 > 0$ such that $R'(x_0)=\frac{p}{q}  < \frac{1}{2}$, our periodic orbit is $\gamma(t)=(x_0, y_0+ \frac{p}{q}t)$.
Therefore, $t=\tau$, $\theta = y_0+ \frac{p}{q}\tau  - \frac{1}{2} t $ and $x_0=|\hat{p}| |\hat{q}|$.
By   relations  $|\hat{p}|^2 -|\hat{q}|^2=1$,  $<\hat{p}, \hat{q}>=0$ and $x_0^2=|\hat{p}|^2  |\hat{q}|^2  $, we get
\begin{equation}\label{eq5}
\begin{split}
&|\hat{p}|^2 =\sqrt{x_0^2 + \frac{1}{4}} + \frac{1}{2}\\
&|\hat{q}|^2 =\sqrt{x_0^2 + \frac{1}{4}} - \frac{1}{2}.
\end{split}
\end{equation}
Write $|\hat{p}|=\frac{e^h+ e^{-h}}{2}$ and  $|\hat{q}|=\frac{e^h-e^{-h}}{2}$.  Follows from the definition, we have
\begin{equation*}
\begin{split}
&p=\cos( \pi t) \hat{p} -\sin (\pi t ) \hat{q}\\
&q=\sin (\pi t) \hat{p} + \cos ( \pi t) \hat{q}.
\end{split}
\end{equation*}
Using the relations $\frac{\hat{p}}{|\hat{p}|} =e^{i 2\pi\theta}$ and $-\hat{q} |\hat{p}| =i |\hat{q} ||\hat{p}|e^{i2\pi\theta}$, we have
\begin{comment}
\begin{equation*}
\begin{split}
&p=\cos( \pi t) \hat{p} -\sin (\pi t ) \hat{q} = e^{i2\pi \theta} \left(   \frac{e^h+ e^{-h}}{2}  \cos (\pi t) +  \frac{e^h -e^{-h}}{2}i \sin (\pi t)\right) = \frac{1}{2} e^h e^{i 2\pi(\theta + \frac{t}{2})} + \frac{1}{2} e^{-h} e^{i 2\pi (\theta - \frac{t}{2})}\\
&q=\sin (\pi t) \hat{p} + \cos ( \pi t) \hat{q}=  -ie^{i 2\pi\theta} \left(   \frac{e^h- e^{-h}}{2}  \cos (\pi t)  +  \frac{e^h +e^{-h}}{2}i \sin  (\pi t) \right) = \frac{i}{2} e^{-h} e^{i 2\pi(\theta - \frac{t}{2})} - \frac{i}{2} e^{h} e^{i2\pi(\theta + \frac{t}{2})}\\
\end{split}
\end{equation*}
\end{comment}
\begin{equation*}
\left\{	\begin{split}
		&p=\cos( \pi t) \hat{p} -\sin (\pi t ) \hat{q}   = \frac{1}{2} e^h e^{i 2\pi(\theta + \frac{t}{2})} + \frac{1}{2} e^{-h} e^{i 2\pi (\theta - \frac{t}{2})}\\
		&q=\sin (\pi t) \hat{p} + \cos ( \pi t) \hat{q} = \frac{i}{2} e^{-h} e^{i 2\pi(\theta - \frac{t}{2})} - \frac{i}{2} e^{h} e^{i2\pi(\theta + \frac{t}{2})}.\\
	\end{split} \right.
\end{equation*}
Then
\begin{equation*}
\left\{\begin{split}
&x_1= \frac{1}{2} e^h e^{i 2\pi (\theta + \frac{t}{2})} +  \frac{1}{2} e^{-h} e^{-i2\pi(\theta - \frac{t}{2})}\\
&x_2=\frac{i}{2} e^{-h} e^{-i 2\pi (\theta - \frac{t}{2})} - \frac{i}{2} e^{h} e^{i2\pi(\theta + \frac{t}{2})}.\\
\end{split} \right. 
\end{equation*}
Replace  $\theta$ by $y_0+ \frac{p}{q}\tau - \frac{1}{2} t $ and $t=\tau$, then we get the result.

%If $x= |\hat{p}| |\hat{q}| <0 $, then $v= \frac{\hat{p}}{|\hat{p}|} = e^{i\theta}$ and   $u=-\hat{q} |\hat{p}|= -|\hat{p}| |\hat{q}| i e^{i\theta}$.  Then
For the case that $ \gamma_{\frac{p}{q}}(\tau)=(\tau, x_0, y_0+ (\frac{p}{q}-1) \tau)$ and $x_0<0$, the argument is similar. We left the details to the reader.
We can regard the  periodic orbit at $x_0=0$ as a limit of the periodic orbits at $x \ne 0$ and $x \to 0$. The periodic orbit at $x_0=0$ is
\begin{equation} 
\left\{
\begin{split}
&x_1= \frac{1}{2}  e^{i (-y_0+ \frac{1}{2} \tau  )} +  \frac{1}{2}  e^{i(y_0 +\frac{1}{2} \tau )}\\
&x_2=\frac{i}{2}  e^{i (-y_0 + \frac{1}{2}\tau )} - \frac{i}{2}  e^{i(y_0 + \frac{1}{2} \tau)}.\\
\end{split}\right.
\end{equation} 
%$x_1(t) = e^{i \frac{t}{2}}$ and $x_2=0$.
\end{proof}

%\subsection{Perturbation}

%\section{Cobordism maps induced by the elementary Lefschetz fibrations }

\section{A formula for the ECH index}\label{section1}
In this section, we deduce a combinatorial  formula for the ECH index as in \cite{H1}, \cite{H5}.  There may be other smarter ways to compute the ECH index, but here we construct several surfaces in $X$ explicitly  and then compute the  ECH index   directly from the  definition.

First, let us consider the ECH index of relative homology classes in $E$. Note that  $H_2(E, \mathbb{Z})=0$,  there is a  unique element $Z_{\alpha}$ in $H_2(E, \alpha)$ for each orbit set $\alpha$.  Therefore, we denote the ECH index,  the relative Chern number, and the relative self--intersection number   by $I(\alpha)$, $c_{\tau}(\alpha)$ and $Q_{\tau}(\alpha)$ respectively. The result on the ECH index  is as follows.
\begin{theorem}\label{thm1}
Let $\pi: E \to D$ be the exact Lefschetz fibration defined in Introduction and orbit set $\alpha=\Pi_i \gamma_{\frac{p_i}{q_i}}$  satisfying $\frac{p_i}{q_i} \ge \frac{p_j}{q_j}$ for $i \le j$.   Then the ECH index is
\begin{equation}
I(\alpha)= Q+ P(Q-P) - \sum_{i<j} (p_iq_j -p_jq_i) - e(\alpha),
\end{equation}
where $e(\alpha)$ is the total multilpity  of elliptic orbits  in $\alpha$, $P=\sum_i p_i$ and $Q=\sum_i q_i$.
\end{theorem}
\begin{proof}
%Since $H_2(E, \mathbb{Z})=0$,  there is an unique element in $H_2(E, \alpha)$. The relative Chern number and relative self-intersection number  are denoted by $c_{\tau}(\alpha)$ and $Q_{\tau}(\alpha)$ respectively.
We compute the quantities $c_{\tau}$, $Q_{\tau}$ and  $CZ_{\tau}$ in Lemmas \ref{lem9}, \ref{lem10} and (\ref{eq6}) respectively. Their proof will appear in the upcoming subsection.
\end{proof}

%Let $c: [0,2 ] \to [0, 1]$ be a function such that $c(t)=t$ when $t \le \frac{1}{2}  $ and $c=1$ when $t \ge 1$.  Define $\kappa: D_2 \to D$ by sending $z$ to $\frac{c(|z|)}{|z|}z$ when $z\ne 0$ and $\kappa(0)=0$. In order to define the symplectic completion, we replace $(E, \omega_E)$ by $\kappa^*(E, \omega_E)$. Then $(E, \omega_E) \vert_{D_2- D}=([1, 2) \times Y, \omega)$.

We can follow \cite{H1} to rewrite the above formula for ECH index in the following way.  Let $w_j=\sum_{i=0}^j(p_i, q_i)$.  Let  $\mathcal{P}(\alpha)$  be the convex path in the plane consisting of straight line segments  between the points $w_{j-1}$ and $w_j$, oriented so that the origin is the initial endpoint. % (Cf. \cite{H1})
Let $\Lambda_{\alpha}$ be the region in the plane which is enclosed by  $\mathcal{P}(\alpha)$  and the line segment from $(0,0) $ to $(P, 0)$ and the line segment from $(P,0) $ to $(P, Q)$.  The area of  $\Lambda_{\alpha}$  is
$$ 2Area(\Lambda_{\alpha}) =PQ  -\sum_{i<j} (p_iq_j -p_jq_i).$$
Therefore, we can rewrite the ECH index   as
\begin{equation}
I(\alpha)= Q+ 2 Area(\Lambda_{\alpha}) -P^2 - e(\alpha).
\end{equation}

\begin{remark}
The   additive property of the ECH index implies that the ECH index of a  relative class in $\mathbb{R} \times \partial_v E$ is   $I(\alpha) -I(\beta)$.  It agrees with the index computation (Proposition 3.2) in \cite{H1}.
\end{remark}
%It is worth noting that $I(\alpha)$ is always positive, thus $\mathcal{M}_0^J(\alpha)$ is empty. Therefore, we have the following vanishing result.
\begin{comment}
\begin{corollary}
Let $\pi: E \to D$ as before and $J$ is a generic almost complex structure.  Then
\begin{equation*}
PFC(E, \Omega_E)_J(\alpha)=0.
\end{equation*}
for any $\alpha \in PFC_*(Y; P, Q)$.
\end{corollary}
\end{comment}

\begin{corollary}
$I(\alpha) \ge 0$ and equality holds if and only if  $\alpha=e_1^m  e_0^n$ for some $m, n \ge 0$.
\end{corollary}
\begin{proof}
Since $0 \le \frac{p_i}{q_i} \le 1$, the triangle determined by $(0, 0)$, $(0, P)$ and $(P, P)$ is inside the region $\Lambda_{\alpha}$. Therefore, $P^2 \le 2 Area(\Lambda_{\alpha})$. By definition, $e(\alpha)  \le Q$. Therefore, $I (\alpha) \ge 0$.  The equality holds if and only if  $Q=e(\alpha)$ and $P^2 = 2 Area(\Lambda_{\alpha})$.  The only  possibility is that  $\alpha=e_1^m  e_0^n$.
\end{proof}

%f $\phi $  consists of the critical points  of $ f $, as well as in each point $x_0$ such that $R'(x_0)=\frac{p}{q}$, there is one elliptic orbit and one hyperbolic orbit as before.
Now let us consider the ECH index of relative homology classes in the ``trivial'' part  $D\times S$. The result is as follows.
\begin{lemma} \label{lem4}
Let  $a$ be a  critical point of $f_S$.  Define $C_a =\{a\} \times D$.  Let $m_a$ be  a non--negative integer.   Then we have  the following three cases:
\begin{itemize}
\item
If $\nabla^2f_S (a)>0$, then the corresponding periodic orbit $e_a$ is elliptic with Conley--Zehnder index $CZ_{\tau}(e_a)=-1$.  In addition, $I(m_a C_a)=0$.
%$I(e_a)={\rm }ind C_a=0$.
\item
  If $\nabla^2f_S (a)<0$, then the corresponding periodic orbit $e_a$ is elliptic with Conley--Zehnder index $CZ_{\tau}(e_a)=1$.  In addition, $I(m_a C_a)=2 m_a$.
 %  Therefore, $I(e_a)={\rm }ind C_a=2$.
 \item
 Finally, if $tr(\nabla^2f_S (a))=0$, then  the corresponding periodic orbit $h_a$ is positive hyperbolic with Conley--Zehnder index $CZ_{\tau}(e_a)=0$.   In addition, $I(m_a C_a)= m_a$. % Therefore, $I(e_a)={\rm }ind C_a=1$.
\end{itemize}
Here  $\tau$ is  the  trivialization  from a fixed  trivialization of  $ TS$ and the standard trivialization of $TD$.
\end{lemma}
\begin{proof}
Let $(x, y)$ be  local coordinates at $a$. %Reintroduce the Morse function $f$ on $S$ which perturbs the admissible 2-form.
It is easy to check that the linearized  Reeb  flow is
\begin{equation}\label{eq20}
\left(
  \begin{array}{ccc}
    \dot{x}   \\
    \dot{y}  \\
  \end{array}
\right)=
 \left(
  \begin{array}{ccc}
    f_{Sxy}(a) & f_{Syy}(a) \\
    -f_{Sxx}(a) & f_{Sxy}(a) \\
  \end{array}
\right)\left(
  \begin{array}{ccc}
    x   \\
    y    \\
  \end{array}
\right).
\end{equation}
By (\ref{eq20}), we know that if   $tr(\nabla^2f_S (a)) \ne 0$,  then it is an elliptic orbit. Otherwise, it is a positive hyperbolic orbit. Also, when  $tr(\nabla^2f_S (a)) \ne 0$, the sign of rotation number of $e_a$ is equal to $-sign( \nabla^2f_S (a) )$.  In conclusion, $e_a$ is $Q$--negative if  $\nabla^2f_S (a) >0$  and  $e_a$ is $Q$--positive if  $\nabla^2f_S (a)<0$,

Finally, it is easy to check that $c_{\tau}(C_a)=1$ and $Q_{\tau}(C_a)=0$ from the definition. These  two ingredients  lead to the statements of the lemma.
\end{proof}

Given an orbit set $\alpha=\Pi_i  \alpha_i^{m_i}$, we define a reference class  ${Z}_{\alpha}=\sum_i m_i[ S_i]  \in H_2(X, \alpha)$,  where $S_i$ are surfaces defined as follows. If $\alpha_i=\gamma_{\frac{p}{q}}$,  then we define a surface $S_i $ to be the image of $u$ in (\ref{eq3}) or (\ref{eq4}) accordingly.  If $\alpha_i$ is a critical point $a$ of $f_S$, then we define $S_i$ by $S_i= \{a\} \times D$. It is worth noting that $u$ doesn't intersect $\{a\} \times D$. Write $\alpha = (\Pi_a \gamma_a^{m_a}) \Pi_i\gamma_{\frac{p_i}{q_i}}$, where $\gamma_a$ are the constant periodic orbits at critical points $a$. Then we have 
\begin{equation*}
I(Z_{\alpha}) = \sum_a I(m_a(\{a\} \times D)) +  I(\Pi_i \gamma_{\frac{p_i}{q_i}}).
\end{equation*}
%Let ${Z}_{\alpha}=\sum_i m_i[ S_i] \in H_2(X, \alpha)$, then $I(Z_{\alpha})=I(\alpha)$.
Recall that $ H_2(X, \alpha)$ is an affine space over $H_2(X, \mathbb{Z})$. By lemma \ref{lem2}, any $Z\in H_2(X, \alpha)$ is of the form $Z=Z_{\alpha} +m [F]$.   Let us denote $I_m(\alpha)=I(Z_{\alpha} +m [F])$.  Therefore, in general we have
%\begin{equation} \label{eq14}
%I(\alpha +m[F])= I(\alpha)+ 2m(Q+1- g(F)).
%\end{equation}
\begin{equation} \label{eq14}
I_m(\alpha )= I_0(\alpha)+ 2m(Q+1- g(F)).
\end{equation}

\subsection{Proof of Theorem \ref{thm1}}
In this subsection, we compute the relative Chern number $c_{\tau}$, the relative self intersection number $Q_{\tau}$ and the  Conley--Zehnder index $CZ_{\tau}$. Theorem \ref{thm1} follows directly from these computations.

\paragraph{Trivialization}
%Recall that the vertical boundary $\partial_v E$ is a mapping torus of Dehn twisted $\phi$, thus $\partial_v E =S^1_t \times [-\lambda, \lambda ]_x \times S^1_y$.
Let us first clarify the trivialization of  what we are using.  Restrict the map $\Phi$ (\ref{eq9}) on the boundary $\partial_v E$, we have a diffeomorphism   $\Phi : \partial_v E  -V \to S^1_t \times (T_{\lambda }-T_0) = S^1_t \times ( [-\lambda,0)\cup (0 , \lambda] ) \times S^1_y $.  The tangent bundle of $T_{\lambda}=[-\lambda, \lambda]_x \times S^1_y$ has a canonical trivialization. Using  $\Phi$ to pull back  this canonical trivialization, we get a trivialization $\tau$ along   each periodic orbit at $x \ne 0$.

%Let $\Phi_{+} $ and $\Phi_-$ be the restriction of $\Phi$ to two component of $\partial_v E - \Sigma$ respectively.   They are not match at $\Sigma$ and thus $\Phi$ cannot be extended to whole $\partial_v E$. But  $\Phi_+ $ and $\Phi_-$ can extend to whole   $\partial_v E $ and give a trivialization $\Phi_{\pm}: \partial_v E \to  S^1_t \times [-\lambda, \lambda]_x \times S^1_y$. Using  $\Phi_{\pm}$ to pull back  the canonical trivialization, they are agree along the periodic orbits at $x_0=0$. We use this as a trivialization $\tau$ along the periodic orbits at $x = 0$.
%Let us write  $\partial_v E - \Sigma= \partial_v E^+ \cup \partial_v E^-$ and $\Phi_{\pm} = \Phi \vert_{ \partial_v E^{\pm}}$.  $\Phi_+$ and $\Phi_-$ do not match at $\Sigma$ and thus $\Phi$ cannot be extended to whole $\partial_v E$.   But  $\Phi_+ $ and $\Phi_-$ can extend to  $\overline{\partial_v E^{\pm}} $ and give  trivializations  $\Phi_{+}: \overline{ \partial_v E^+} \to  S^1_t \times [0, \lambda]_x \times S^1_y$ and  $\Phi_{-}: \overline{ \partial_v E^- }\to  S^1_t \times [-\lambda, 0]_x \times S^1_y$ respectively. Using  $\Phi_{\pm}$ to pull back  the canonical trivialization, they agree along the periodic orbits at $x_0=0$. We use this as a trivialization $\tau$ along the periodic orbits at $x = 0$.
Let us write  $\partial_v E - V= \partial_v E^+ \cup \partial_v E^-$ and $\Phi_{\pm} = \Phi \vert_{ \partial_v E^{\pm}}$.  $\Phi_+$ and $\Phi_-$ do not match at $V$ and thus $\Phi$ cannot be extended to the  whole $\partial_v E$.   But  $\Phi_+ $ and $\Phi_-$ can extend to  $\overline{\partial_v E^{\pm}} $ and give  trivializations  $\Phi_{+}: \overline{ \partial_v E^+} \to  S^1_t \times [0, \lambda]_x \times S^1_y$ and  $\Phi_{-}: \overline{ \partial_v E^- }\to  S^1_t \times [-\lambda, 0]_x \times S^1_y$ respectively. Using  $\Phi_{\pm}$ to pull back  the canonical trivialization, hence there are two trivializations $\tau_{\pm}$ along the periodic orbits at $x_0=0$. But there is no difference between using $\tau_+$ and $\tau_-$ when we compute $c_{\tau}, Q_{\tau}, CZ_{\tau}$. So we just use the  same notation $\tau$ to denote one of them.

\paragraph{Conley--Zehnder index }
Since $R''(x)< 0$,   the elliptic orbit $e_{\frac{p}{q}}$  has small negative rotation number with respect to the trivialization $\tau$. %Then  it is $Q$-negative elliptic
As long as the perturbation  $f_{T_{x_0}}$ is  small enough,  we have
\begin{equation}\label{eq6}
CZ_{\tau}(e^k_{\frac{p}{q}}) = -1,  \ \ \ CZ_{\tau}(h_{\frac{p}{q}}) = 0,
\end{equation}
for any $k \le Q$.

\paragraph{Relative intersection number }
%The main result of this section is the following lemma.
The result about the relative self--intersection is as follows:
\begin{lemma}\label{lem9}
For $0 \le  \frac{p}{q}, \frac{p'}{q'} \le 1$, then $Q_{\tau}(\gamma_{\frac{p}{q}}, \gamma_{\frac{p'}{q'}}) =  \min\{ p(q'-p'), p'(q-p)\}$.  Assume that $\alpha=\Pi_i \gamma_{\frac{p_i}{q_i}}$ and $\frac{p_i}{q_i} \ge \frac{p_j}{q_j}$ for $i \le j$. Then
\begin{equation*}
Q_{\tau}(\alpha)=P(Q-P)  - \sum_{i<j}( p_iq_j -p_jq_i),
\end{equation*}
where $P =\sum_ip_i$ and $Q=\sum_i q_i$.
\end{lemma}
It is worth noting that  $Q_{\tau} $ is quadratic in the  sense that
\begin{equation*}
Q_{\tau}(\alpha)=\sum_i Q_{\tau}(\gamma_\frac{p_i}{q_i}) + 2\sum_{i<j} Q_{\tau}(\gamma_\frac{p_i}{q_i}, \gamma_\frac{p_j}{q_j}).
\end{equation*}
Turn out it suffices to compute   $ Q_{\tau}(\gamma_\frac{p_i}{q_i}) $ and $Q_{\tau}(\gamma_\frac{p_i}{q_i}, \gamma_\frac{p_j}{q_j}).$
The idea of  computing   $ Q_{\tau}(\gamma_\frac{p_i}{q_i}) $ and $Q_{\tau}(\gamma_\frac{p_i}{q_i}, \gamma_\frac{p_j}{q_j})$   is to express them  as intersection numbers of two surfaces.  To this ends, let us construct  a  surface $u: D_z \to E$ which  is asymptotic to $\gamma_{\frac{p}{q}}$ as follows.

 Let $\epsilon_i(r) : [0 , \infty) \to  [0, \infty)$ be   cutoff functions  with the following properties:
\begin{itemize}
\item
$\epsilon_i(r)$ is nonincreasing and it is  supported  in $r \le 2\delta$.
\item
 When $r \le \delta$,   $\epsilon_i(r)$ is constant. The constant is still denoted by $\epsilon_i$ and $ \epsilon_i \ll \delta$.
\end{itemize}
If $0\le \frac{p}{q}\le\frac{1}{2}$,   define $u(z)$ by 
\begin{equation} \label{eq3}
\left(  (1+ \epsilon_1)\frac{1}{2}e^{h+iy_0} z^p + (1+ \epsilon_2)\frac{1}{2}e^{-h-iy_0} z^{q-p} + \epsilon_3,   -(1+ \epsilon_1)\frac{i}{2}e^{h+iy_0} z^p + (1+ \epsilon_2)\frac{i}{2}e^{-h-iy_0} z^{q-p} \right).
\end{equation}
If $\frac{1}{2}\le\frac{p}{q} \le 1$,  define $u(z)$ by 
\begin{equation}\label{eq4}
 \left(  (1+ \epsilon_1)\frac{1}{2}e^{h-iy_0} z^{q-p} + (1+ \epsilon_2)\frac{1}{2}e^{-h+iy_0} z^{p} + \epsilon_3,   (1+ \epsilon_1)\frac{i}{2}e^{h-iy_0} z^{q-p} - (1+ \epsilon_2)\frac{i}{2}e^{-h+iy_0} z^p \right).
\end{equation}
Note that these two definitions coincide when $\frac{p}{q}=\frac{1}{2}$.  By construction,  $u$ is asymptotic to $\gamma_{\frac{p}{q}}(t)$ as $z $ tends to boundary of $D$.

\begin{comment}
Let $x_0>0$ such that $R'(x_0)=\frac{p}{q}< \frac{1}{2}$.  The the periodic orbit $\gamma(t)$ is described by \ref{eq16}. Define $u: D_z \to E$ by
\begin{equation*}
u(z)=\left(  \frac{1}{2}e^h z^p + \frac{1}{2}e^{-h} z^{q-p},   -\frac{i}{2}e^h {z}^p +  \frac{i}{2}e^{-h} z^{q-p} \right).
\end{equation*}
Obviously, as $z $ tends to boundary of $D$, $u$ is asymptotic to $\gamma(t)$.  Let $\epsilon_i(r)$ be a small bump function which is support in $r \le 2\delta$ and $\epsilon_i(r)=\epsilon_i$ when $r \le \delta$.  We can perturb $u$ by $\epsilon$,
\begin{equation*}
u(z)=\left(  (1+ \epsilon_1)\frac{1}{2}e^h z^p + (1+ \epsilon_2)\frac{1}{2}e^{-h} z^{q-p} + \epsilon_3,   -(1+ \epsilon_1)\frac{i}{2}e^h z^p + (1+ \epsilon_2)\frac{i}{2}e^{-h} z^{q-p} \right).
\end{equation*}
For $\frac{p'}{q'} < \frac{p}{q}$, we define
\begin{equation*}
v(w)=\left(  (1+ \epsilon_4)\frac{1}{2}e^k w^{p'} + (1+ \epsilon_5)\frac{1}{2}e^{-k} w^{q'-p'},   -(1+ \epsilon_4)\frac{i}{2}e^k {w}^{p'} + (1+ \epsilon_5)\frac{i}{2}e^{-k} w^{q'-p'} \right).
\end{equation*}
\end{comment}

\begin{lemma} \label{lem11}
The map $u$  satisfies the following properties:
\begin{enumerate}
\item
For sufficiently small $\epsilon_i$ and $\delta$,  then $u$ is embedded except at $z=0$.
\item
%All the intersection points of $u$ and $v$ lies in the region $\{0<|z| \le \delta\} \times \{|w| \le \delta\}  $.
Let $\frac{p'}{q'} < \frac{p}{q}$ and  $v$ be the $\frac{p'}{q'} $--version of (\ref{eq3}) or (\ref{eq4}) accordingly. If $v$ doesn't involve the zero--order term, i.e., $\epsilon_3=0$, then for sufficiently small $\epsilon_i$ and $\delta$, the intersection points of $u$ and $v$ lie  in the region $\{0<|z| \le \delta\} \times \{0<|w| \le \delta\}  $.
\item
The intersection of $u$ and $v$ are  transversal and the sign of the intersection points are positive.
\end{enumerate}
\end{lemma}
\begin{proof}
\begin{enumerate}
\item
It is straightforward  to check that $u$ is immersed except at $z=0$. Moreover, $u$ is 1--1 onto its image for sufficiently small $\epsilon_i$.  To see this,  note that the unperturbed version of $u$ ($\epsilon_i=0$) is 1--1 onto its image  because  $p$ and $q$ are  relatively prime. By using the limit argument and the fact that $u$ is immersion, we can deduce the same conclusion for  sufficiently small $\epsilon_i$.

%if there are $z_1 \ne z_2$ such that $u(z_1) =u(z_2)$,  one can use the argument in the proof of second item  to show that $(z_1, z_2)$ is outside the region $\{ \delta< |z| < 2\delta\} \times \{ \delta< |z| < 2\delta\}$. Thus $\epsilon_i(|z_1|)=\epsilon_i(|z_2|)$ are constant.  Then we will get $z_1^p =z_2^p$ and $z_1^q=z_2^q$ and this implies $z_1=z_2$,  because, $p$ and $q$ are  relative prime.

\item
There are three possibilities  of the order of $\frac{p}{q}, \frac{p'}{q'}$ :  $0 \le \frac{p'}{q'} < \frac{p}{q} \le \frac{1}{2}$,  $\frac{1}{2} \le \frac{p'}{q'}< \frac{p}{q} \le1$ and $ \frac{p'}{q'}< \frac{1}{2}< \frac{p}{q} \le1$.  We prove the statement case by case.  Without loss of generality, we assume that $y_0=0$ in (\ref{eq3}), (\ref{eq4}).

Let us consider the case that $0 \le \frac{p'}{q'} < \frac{p}{q} \le \frac{1}{2}$. Let $(z, w) $ be an intersection point  of $u$ and $v$.  It satisfies  the equations
\begin{equation*} \left\{
\begin{split}
&(1+ \epsilon_1)\frac{1}{2}e^h {z}^p + (1+ \epsilon_2)\frac{1}{2}e^{-h} z^{q-p} + \epsilon_3= (1+ \epsilon_4)\frac{1}{2}e^k  {w}^{p'} + (1+ \epsilon_5)\frac{1}{2}e^{-k} w^{q' -p'}\\
& -(1+ \epsilon_1)\frac{i}{2}e^h  {z}^p + (1+ \epsilon_2)\frac{i}{2}e^{-h} z^{q-p}= - (1+ \epsilon_4)\frac{i}{2}e^k {w}^{p'} + (1+ \epsilon_5)\frac{i}{2}e^{-k} w^{q'-p'}.
\end{split} \right.
\end{equation*}
These are equivalent to
\begin{equation} \label{eq11}
	\left\{
\begin{split}
&(1+ \epsilon_1)e^h {z}^p  + \epsilon_3= (1+ \epsilon_4)e^k {w}^{p'} \\
&  (1+ \epsilon_2)e^{-h} z^{q-p} +  \epsilon_3=   (1+ \epsilon_5)e^{-k} w^{q'-p'}.
\end{split} \right. 
\end{equation}
Assume that $0< \epsilon_i \ll \delta \ll1 $.  Then the  solutions to  (\ref{eq11})  lie inside either $\{|z| , |w| \le \delta\}$ or $\{|z|, |w| \ge c_0^{-1}\}$ for some constant $c_0 \ge 1$ and $c_0^{-1} \gg \delta$. To see this, if $|w| > \delta$, then (\ref{eq11}) implies
\begin{equation*}
\begin{split}
& |z|^p \ge \frac{1}{1+\epsilon_1 } \left((1+ \epsilon_4) e^{k-h}|w|^{p'} -\epsilon_3 e^{-h}\right) \ge c_0^{-1}|w|^{p'};\\
& |z|^{q-p}  \le \frac{1}{1+\epsilon_2} \left( (1+ \epsilon_5)e^{h-k}|w|^{q'-p'} +\epsilon_3 e^{h}\right) \le c_0|w|^{q'-p'}.
\end{split}
\end{equation*}
Therefore,  $|w|^{pq'-p'q} \ge c_0^{-q}$.  Similarly,  if $|z| > \delta$, then we can deduce that $|z| \ge c_0^{-1} $.  Note that the cases $\{|z| \le \delta, |w| \ge c_0^{-1}\}$ and $\{|w| \le \delta, |z| \ge c_0^{-1}\}$  cannot happen by  (\ref{eq11}).

\begin{comment}%%%%%%%%%%%%%%%%%%%%There is another good argument.
If $\delta<|z|< 2\delta$, then
$$|w|^{p'} \le e^{h-k} |z|^p + \epsilon_3' \le e^{h-k} 2^p \delta^p + \epsilon_3' .$$
On the other hand,
$$|w|^{q'-p'} \ge e^{k-h} |z|^{q-p} - \epsilon_3'' \ge e^{k-h}  \delta^{q-p} - \epsilon_3'' .$$ Therefore, we have
$$e^{k-h}  \delta^{q-p} - \epsilon_3'' \le \left( e^{h-k} 2^p \delta^p + \epsilon_3' \right)^{\frac{q'-p'}{p'}}.$$
However, this is impossible for small $\delta$ and $\epsilon_i$.

Similarly for the case that $\delta<|w| < 2\delta$.  We also need to rule out the case that $(z, w)$ lies inside $\{0\le |z| \le \delta, |w| \ge 2 \delta\}$, $\{0\le |w| \le \delta, |z| \ge 2 \delta\}$ and $\{|z|, |w| \ge 2 \delta\}$. For the first case, note that $\{\epsilon_i\}_{i=1}^3$ are constant and $\epsilon_4=\epsilon_5=0$. Assume that $\epsilon_i \ll \delta$. By \ref{eq11}, we have
\begin{equation*}
\begin{split}
& |z|^p \ge \frac{1}{1+\epsilon_1 } \left( e^{k-h}|w|^{p'} -\epsilon_3 e^{-h}\right) \ge c_0^{-1}|w|^{p'}\\
& |z|^{q-p}  \le \frac{1}{1+\epsilon_2} \left( e^{h-k}|w|^{q'-p'} -\epsilon_3 e^{h}\right) \le c_0|w|^{q'-p'}.
\end{split}
\end{equation*}
This implies that $|w| \ge c_0^{-1}$. On the other hand, $|w|^{p'}  \le c_0|z|^p + \epsilon_3 e^k \le c_0 \delta^p$. Take $0<\delta \ll 1$, then we get contradiction.  For the second case, the argument is similar.
\end{comment}

In the case that $|w|, |z| \ge c_0^{-1} >2\delta$, then $\epsilon_i=0$ and  (\ref{eq11}) becomes
\begin{equation*}
	\left\{
\begin{split}
&e^h {z}^p  = e^k {w}^{p'} \\
&  e^{-h} z^{q-p} =   e^{-k} w^{q'-p'}.
\end{split} \right. 
\end{equation*}
%Take absolute value both side and one can solve that the norm of $|z|$ and $|w|$ only dependent on $p,q , k ,h$.  In fact,
By the above equations, we have
\begin{equation*}
\log|z|= \frac{k-h}{pq'-p'q}q',  \ \ \  \log|w|=\frac{k-h}{pq'-p'q}q.
\end{equation*}
%By directly computation, $k  >  h$ if $\frac{p}{q} > \frac{p'}{q'}$.
Assume that $R'(x_0)=\frac{p}{q}$ and $R'(x_1)=\frac{p'}{q'}$. Then $0<x_0< x_1$  because  $R''<0$.  By equations (\ref{eq5}),  it is easy to check that $k>h$.
Then the  norms   $|z|, |w|>1$ which are not in our domain.
% Finally, $(0,0)$ obviously is not an intersection point.
 Finally,  $(0,0)$, $(z, 0)$ and $(0, w)$ cannot be   intersection points for suitable  choices of $\epsilon_i$. This can be checked  directly.

For the case that $\frac{1}{2} \le \frac{p'}{q'}< \frac{p}{q} \le1$,  we get
\begin{equation} \label{eq12}
	\left\{ 
\begin{split}
&(1+ \epsilon_1)e^h {z}^{q-p}  + \epsilon_3= (1+ \epsilon_4)e^k {w}^{q'-p'} \\
&  (1+ \epsilon_2)e^{-h} z^{p} +  \epsilon_3=   (1+ \epsilon_5)e^{-k} w^{p'}. 
\end{split}\right. 
\end{equation}
 at the intersection point.
The same argument can show that the solutions to  (\ref{eq12})  lie inside either $\{|z| , |w| \le \delta\}$ or $\{|z|, |w| \ge c_0^{-1}\}$ for some constant $c_0 \ge 1$ and $c_0^{-1} \gg \delta$.  Note that  $\epsilon_i=0$ where  $|w|, |z| \ge c_0^{-1} > 2\delta$. If $|w|, |z| \ge c_0^{-1}$, then  (\ref{eq12}) implies that
%Take absolute value both side of \ref{eq12} and one can solve that
\begin{equation*}  
\log|z|= \frac{h-k}{pq'-p'q}(q'),  \ \ \  \log|w|=\frac{h-k}{pq'-p'q}(q).  
\end{equation*}
%By directly computation, $h>k$ if $\frac{p}{q} > \frac{p'}{q'}$.
Assume that $R'(x_0)=\frac{p}{q}$ and $R'(x_1)=\frac{p'}{q'}$.  Then $x_0< x_1<0$  because of $R''<0$.  By the definition of $h$ and $k$,  it is easy to check that $h>k$.
Then the  norms  $|z|, |w|>1$  which are not in our domain.  % Finally, $(0,0)$ obviously is not an intersection point.

Again,  $(0,0)$, $(z, 0)$ and $(0, w)$ cannot be   intersection points for suitable  choices of $\epsilon_i$.

For the case that   $ \frac{p'}{q'}< \frac{1}{2}< \frac{p}{q} \le1$, we have
\begin{equation} \label{eq13}
	\left\{ 
\begin{split}
&(1+ \epsilon_1 )e^h z^{q-p} + \epsilon_3 = (1+ \epsilon_5) e^{-k}w^{q'-p'}\\
&  (1+ \epsilon_2 )e^{-h} z^{p} + \epsilon_3 = (1+ \epsilon_4) e^{k}w^{p'}\\
\end{split} \right. 
\end{equation}
The same argument can show that the solutions to  (\ref{eq13})  lie inside either $\{|z| , |w| \le \delta\}$ or $\{|z|, |w| \ge c_0^{-1}\}$ for some constant $c_0 \ge 1$ and $c_0^{-1} \gg \delta$. % If $|w|, |z| \ge c_0^{-1} > 2\delta$, then $\epsilon_i=0$.
%Take absolute value both side and one can solve that
As before, if $|w|, |z| \ge c_0^{-1} > 2\delta$, then they satisfy
\begin{equation*}
\log|z|= \frac{h+k}{pq'-p'q}(q'),  \ \ \  \log|w|=\frac{h+k}{pq'-p'q}(q).
\end{equation*}
Again they are not in our domain.

The intersection points of the forms  $(0,0)$, $(z, 0)$ and $(0, w)$ can  be ruled out as the other cases.     %intersection points for suitable  choice of $\epsilon_i$. This can be check directly.
\item
In the region $\{0<|z| \le \delta\} \times \{0<|w| \le \delta\}  $,  $\epsilon_i$ are constants. The statement follows from that the coordinate functions $x_1$ and $x_2$ are holomorphic with respect to $z$ and $w$.
\end{enumerate}
\end{proof}

Let us consider the case that $\frac{p}{q} =\frac{p'}{q'}$.  Without loss of generality, we only consider the case that $x_0>0 $ and $R'(x_0)=\frac{p}{q}< \frac{1}{2}$. The argument for the other cases are the same, we left the details to the reader. Let
\begin{equation*}
u(z)=\left(  (1+ \epsilon_1)\frac{1}{2}e^h z^p + \frac{1}{2}e^{-h} z^{q-p} + \epsilon_2,   -(1+ \epsilon_1)\frac{i}{2}e^h z^p + \frac{i}{2}e^{-h} z^{q-p} \right)
\end{equation*}
and
\begin{equation*}
v(z)=\left(  \frac{1}{2}e^{h+iy_0} w^p + \frac{1}{2}e^{-h-iy_0} w^{q-p} ,   -\frac{i}{2}e^{h+iy_0} z^p + \frac{i}{2}e^{-h-iy_0} w^{q-p} \right).
\end{equation*}

\begin{lemma}
The maps $u$ and $v$ satisfy the following properties:
\begin{enumerate}
\item
For sufficiently small $\epsilon_i$ and $\delta$,  then  $u$ and $v$ are embedded except at $z=0$.
\item
Assume that $\epsilon_1 \gg \epsilon_2$ and $y_0$ is generic. Then all the intersection points of $u$ and $v$ lie in the region $\{0<|z| \le \delta\} \times \{0<|w| \le  2\delta\}  $.
\item
The intersection of $u$ and $v$ are  transversal and the sign of the intersection points are positive.
\end{enumerate}
\end{lemma}
\begin{proof}
We only prove the second statement and the proof of the other two is the same as in Lemma \ref{lem11}.   Let $(z, w)$ be an  intersection point. Then  it satisfies equations
\begin{equation*}
	\left\{ 
\begin{split}
&(1+ \epsilon_1)\frac{1}{2}e^h {z}^p + \frac{1}{2}e^{-h} z^{q-p} + \epsilon_2= \frac{1}{2}e^{h+iy_0}  {w}^{p} + \frac{1}{2}e^{-h-iy_0} w^{q -p}\\
& -(1+ \epsilon_1)\frac{i}{2}e^h  {z}^p + \frac{i}{2}e^{-h} z^{q-p}= - \frac{i}{2}e^{h+iy_0} {w}^{p} +  \frac{i}{2}e^{-h-iy_0} w^{q-p}.
\end{split} \right. 
\end{equation*}
The above equations are equivalent to
\begin{equation}\label{eq10}
	\left\{  
\begin{split}
&(1+ \epsilon_1)e^h {z}^p  + \epsilon_2= e^{h+iy_0}{w}^{p} \\
&  e^{-h} z^{q-p} +  \epsilon_2=   e^{-h-iy_0} w^{q-p}.
\end{split}\right. 
\end{equation}
If $|z| \ge 2\delta$, then we have $\epsilon_i=0$, $z^p=e^{iy_0}w^{p}$ and $z^q=w^q$.  Using  the polar coordinates, it is easy to check that $\frac{p}{q}= \frac{y_0+ 2\pi k }{2\pi l} $ for some $k, l \in \mathbb{Z}$.  This is impossible for a generic $y_0$.  Thus there are no intersection points in the region $|z| \ge 2\delta$.

If $\delta \le |z| < 2\delta$, then   (\ref{eq10})  implies that
\begin{equation*}
\begin{split}
&(1+ \epsilon_1) |z|^p - \frac{1}{2}e^{-h} \epsilon_2 \le |w|^p; \\
&|w|^{q-p} \le |z|^{q-p} + e^{h} \epsilon_2.
\end{split}
\end{equation*}
Hence,
\begin{equation*}
\begin{split}
(1+ \epsilon_1) |z|^p \le |z|^p \left(  1 + e^h  \frac{\epsilon_2}{|z|^{q-p}} \right)^{\frac{p}{q-p}} + e^{-h} \epsilon_2  \le |z|^p +  c_0\frac{e^h \epsilon_2}{ |z|^{q-2p} } +e^{-h} \epsilon_2.
\end{split}
\end{equation*}
Therefore,
$$\delta^p \epsilon_1 \le  c_0 \delta^{2p-q} \epsilon_2. $$
This is impossible provided that we choose $\epsilon_1 =100 c_0 \delta^{p-q} \epsilon_2$. In conclusion, the only possibility  is that $0 \le |z| < \delta$. From  (\ref{eq10}), it is easy to deduce that $0 \le |w| \le 2 \delta$.

Again,  it is straightforward to check that  $ (0,0)$, $(z, 0)$ and $(0, w)$ cannot be   intersection points.
\end{proof}

\begin{proof}[Proof of Lemma \ref{lem9}]
For $\frac{p}{q}  \ne \frac{p'}{q'}$, then $Q_{\tau}(\gamma_{\frac{p}{q}},  \gamma_{\frac{p'}{q'}} )$   equals to the intersection number of $u$ and $v$.  The  intersection number  equals to the number of solutions to equations   (\ref{eq11}) or (\ref{eq12}) or (\ref{eq13}) which lie inside the region $\{0<|z| \le \delta\} \times \{0<|w| \le \delta\}  $.   Keep in mind that our $\epsilon_i$ are constants in this region rather than functions. Replace the functions $\epsilon_i$ in (\ref{eq11}), (\ref{eq12}) and  (\ref{eq13})  by constants.  
According to  the theorem of Bernshtein \cite{DNB} (also see Chapter 3 of \cite{BS}),  equations   (\ref{eq11}) or (\ref{eq12}) or (\ref{eq13})  have total     $ \max\{ p(q'-p'), p'(q-p)\}$     solutions  in $\mathbb{C}^* \times \mathbb{C}^*$.   However, there are $|pq'-qp'|$ solutions outside the region   $\{0<|z|, |w|< \delta\}$. In fact,  the estimates in the proof of  Lemma \ref{lem11} implies that these solutions lie in the region $\{1\le |z|, |w|\}$.    To see this, note that the unperturbed equations  (\ref{eq11}) or (\ref{eq12}) or (\ref{eq13}) ($ \epsilon_i=0$) have exactly $|pq'-qp'|$ solutions in the region $\{1< |z|, |w|\}$.  By the limit argument, the same conclusion holds provided  that $\epsilon_i$   are  small enough. In conclusion, we have
 $$ Q_{\tau}(\gamma_{\frac{p}{q}}, \gamma_{\frac{p'}{q'}}) =  \max\{ p(q'-p'), p'(q-p)\} -|pq'-p'q|= \min\{p(q'-p'), p'(q-p)\}.$$

Note that $Q_{\tau}(e_{\frac{p}{q}}) =Q_{\tau}(h_{\frac{p}{q}})  =Q_{\tau}(e_{\frac{p}{q}}, h_{\frac{p}{q}})   $ with respect to our trivializations.  So the relative self intersection number   equals to the number of solutions of  (\ref{eq10}) in the region   $\{0<|z|, |w|< \delta\}$. By Bernshtein's theorem \cite{DNB} and the same argument as before,  we have $$Q_{\tau}(e_{\frac{p}{q}}, h_{\frac{p}{q}})  = p(q-p). $$

\end{proof}
%%%%%%%%%%%%%%%%%%%%%%%%%%%%%%%%%%%%%%%%%%%%%%%%%%%%%%%%%%%%%

\paragraph{Relative Chern number }
%The main result of this section is the following lemma:
\begin{lemma} \label{lem10}
Let $\alpha =\Pi_i \gamma_{\frac{p_i}{q_i}}$ be an orbit set  with $P=\sum_i p_i$ and $Q=\sum_i q_i$. Then  we have $c_{\tau} (\alpha)=Q$.
\end{lemma}

\begin{proof}
Note that  the relative Chern number  satisfies the additive  property
\begin{equation*}
c_{\tau}(\alpha)=\sum_i c_{\tau}(\gamma_{\frac{p_i}{q_i}}).
\end{equation*}
Thus it suffices to compute $c_{\tau}(\gamma_{\frac{p}{q}})$.

By definition, we have
\begin{equation*}
4x^2=4|\hat{p}|^2|\hat{q}|^2=(|x_1|^2 + |x_2|^2)^2 - |\pi(x_1, x_2)|^2=2|x_1|^2 |x_2|^2 - \bar{x}_1^2 x_2^2 - x_1^2 \bar{x}_2^2.
\end{equation*}

By differentiating   both side of the above identity, we obtain
\begin{equation*}
4xdx=(x_1\bar{x}_2-\bar{x}_1 x_2) (x_2 d \bar{x}_1-\bar{x}_2 dx_1 + \bar{x}_1 dx_2 - x_1d\bar{x}_2).
\end{equation*}
Note that $4x^2 =-(x_1\bar{x}_2-\bar{x}_1 x_2 )^2=4 (Im(x_1 \bar{x}_2))^2$. Along the periodic orbits at $x \ne 0$, we can check that $x=Im(x_1 \bar{x}_2)$. Therefore,
\begin{equation*}
dx= \frac{i}{2} (x_2 d \bar{x}_1-\bar{x}_2 dx_1 + \bar{x}_1 dx_2 - x_1d\bar{x}_2)
\end{equation*}
along the periodic orbit. In addition, it can be defined along the periodic orbit at $x=0$.

Let $J \in \mathcal{J}_h(E, \omega_E)$ be an almost complex structure such that $J(\partial_x) =\partial_y$ over  $1-\delta \le r \le 1$ and $J=J_0$ near the critical point.  Therefore, $dx + idy= dx + iJdx {\pm} iR_r' ( \pm x) dt$ over the region  where $x \ne 0$ and $1- \delta \le r\le 1$.  $T^{1, 0}_J E $ is generated by
\begin{equation*}
\begin{split}
ds+ idt, \ \ \ \ dx+iJdx {\pm}\frac{R_r'({\pm}x)}{2}(ds +idt).
\end{split}
\end{equation*}
Define a section $\psi= (ds + idt )\wedge ( dx+iJdx  \pm \frac{R_r'(\pm x)}{2}(ds +idt)) = \frac{\bar{z}dz}{|z|^2}\wedge ( dx+iJdx )$, where $z = \pi( \textbf{x})$.
%Note that this section can extend over the thimble.
Near  the critical point of $\pi_E$,
$ dx+iJdx  =  dx+iJ_0dx= -i \bar{x}_2 dx_1 +i \bar{x}_1 dx_2$.  Thus
\begin{equation*}
\begin{split}
\psi=\frac{2 \bar{z}}{|z|^2}(x_1dx_1 + x_2 dx_2)\wedge (-i \bar{x}_2 dx_1 +i \bar{x}_1 dx_2) = 2i\bar{z} dx_1 \wedge dx_2.
\end{split}
\end{equation*}
Therefore, we can extend $\psi$ to a section of $T^{2, 0}_J E$ over the whole $E$.

Obviously, $\psi$ is  nowhere vanishing except at $z=0$.  Let $u: D \to E$ be the unperturbed version of  (\ref{eq3}) or (\ref{eq4}). Then  $\psi \vert_u$ vanishes  at the critical point and the vanishing order is $-q$. Therefore, $c_{\tau}(\gamma_{\frac{p}{q}})=- \#\psi^{-1}(0)=q$.

\end{proof}

\section{Proof of the main results}

\subsection{Energy constraint}
Before we prove Theorem \ref{thm3}, let us  write down a formula for the $\omega_X$--energy and deduce  a constraint on the relative homology class.

Let  $Z_{\alpha} \in H_2(E, \alpha)$. Its  $\omega_E$--energy is denoted by  $E(\alpha)$.  By  Stokes’ theorem,  $E(\alpha)=\int_{\alpha} \theta_E$.  %Now let us compute an explicit formula  for  the energy.
By  a direct  computation, we have
\begin{equation*}
E(\gamma_{\frac{p}{q}})  = \left\{
\begin{aligned}
& x_0p -R(x_0 )q \   \mbox{ at  $ x \ge  0$ such that $R'(x_0)=\frac{p}{q}$} \\
&  x_0(p-q) - R(-x_0)q.  \   \mbox{ at $ x \le 0$ such that $R'(x_0)=\frac{p}{q}$} \\
\end{aligned}
\right.
\end{equation*}
%At $x_0 \ge 0$ such that $R'(x_0)=\frac{p}{q}$, we have
%\begin{equation*}
%E(\gamma_{\frac{p}{q}}) =\int_0^q \left(x_0R'(x_0)   - R(x_0) \right)dt =x_0p -R(x_0 )q .
%\end{equation*}
% At $x_0 \le 0$ where $R'(-x_0)=1-\frac{p}{q}$, we have
% \begin{equation*}
% E(\gamma_{\frac{p}{q}}) =\int_0^q \left(-x_0R'(-x_0)   - R(-x_0) \right)dt =-x_0R'(-x_0)q - R(-x_0 )q = x_0(p-q) - R(-x_0)q.
%\end{equation*}
%At $x_0=0$, $E(\gamma_{\frac{1}{2}})=0$, because,  we can choose a surface $Z$ which is asymptotic to $\gamma_{\frac{1}{2}}$ and $Z$ lies in side the symplectic thimble $\Sigma$.
Therefore,
 \begin{equation*}
 E(\alpha) = \sum_i   \left( |x_i| R'(|x_i|) -R(|x_i|) \right) q_i.
\end{equation*}
By the definition of $R(|x|)$, it is easy to show that $E(\alpha) \le {Q}$.  Note that  we have $E(S_a)=0$ for  $S_a=\{a\} \times D$.  Thus for $Z_{\alpha} \in H_2(X, \alpha)$, we still have $E(Z_{\alpha}) \le {Q}$. After  being perturbed  by $f_S$ and $\{f_{T_{x_0}}\}$, the energy $E(Z_{\alpha}) \le (1 + \varepsilon) Q$ for a small $\varepsilon>0$.
\begin{remark} \label{rmk2}
Take $f_{T_{x_0}}$ such that the minimum is $0$. Then after the Morse--Bott perturbations, we  have  $E(e_{\frac{p}{q}}) =E(\gamma_{\frac{p}{q}})$.
\end{remark}

For a general class $Z_{\alpha}+ m[F] \in H_2(X, \alpha)$, the $\omega_X$--energy is
\begin{equation} \label{eq21}
 E(Z_{\alpha} + m[F])  = E(Z_{\alpha}) + m \int_{F} \omega_X.
\end{equation}
\begin{remark} \label{rmk1}
By combining (\ref{eq21}) and  (\ref{eq14}), we know that  the ECH index and the  energy satisfy a  relation $$I(Z_1)-I(Z_2)=  c(E(Z_1) -E(Z_2))$$ for some constant $c$.  Note that $c\ne 0$ if $Q \ne g(F)-1$.  This property ensures that the right hand side of (\ref{eq2})  is a finite sum.  Thus the cobordism maps can be defined with $\mathbb{Z}_2$-- or $\mathbb{Z}$--coefficients.
\end{remark}

\begin{definition} [See \cite{PS}, \cite{PS1}.]
Fix $J \in\mathcal{J}_h(W, \omega_W) $.  A tuple $(\pi_W: W \to B, \omega_W, J) $ is called   non--negative  if   at any point $w $ away from the critical points of $\pi_W$, then     $\omega_W \vert_{T W_w^{hor}} = \rho(w) (\pi_W^* \omega_B)  \vert_{TW_w^{hor}} $ , where   $\omega_B \in \Omega^2(B)$  is a positive volume form  and $\rho$ is a non--negative function.
\end{definition}

\begin{remark}\label{rmk3}
An important consequence of this definition is that  if   $(\pi_W: W \to B, \omega_W, J) $ is  non--negative, then any $J$--holomorphic current $\mathcal{C}$ has non--negative $\omega_W$--energy. The reason is that  $$E(C) =\int_C |TC^{vert}|^2+ \int_{\pi_W(C)} \rho \vert_C \omega_B$$ for every  simple holomorphic curve  $C$ away from the critical points of $\pi_W$, where $TC^{vert}$ is the vertical component of $TC$. Moreover, the holomorphic curve with  ``minimal'' energy is very rigid in the following sense:
   If $E(C) = \int_{\pi_W(C)} \rho \vert_C \omega_B$, then $C$ is  horizontal, i.e., $TC \subset T\overline{W}^{hor}$.

   The elementary Lefschetz fibration $(X, \pi_X, \omega_X)$ together with $J \in J_h(X, \omega_X)$ is non--negative.  (See Lemmas 18.3, Proposition 19.10 of \cite{PS1}.) Also, the tuple $(X, \pi_X, \omega_X, J)$ is still non--negative after being  perturbed by $f_S, \{f_{T_{x_0}}\}$ as long as  these functions are non--negative.
\end{remark}

\begin{lemma} \label{lem3}
Given a positive integer  $Q$,  if     $\int_F \omega_X \ge Q+1$, then we have  the following property:  Fix $J \in \mathcal{J}_h(X, \omega_X)$. Then for any orbit set $\alpha$ with degree less than or equal to  $Q$  and $m<0$, there is no holomorphic current $\mathcal{C}$ with  relative homology class $Z_{\alpha}+ m[F]$.
\end{lemma}
\begin{proof}
According to the above discussion,  we know that  $E(\alpha + m[F])  \le (1+ \varepsilon) {Q} + m \int_{F} \omega_X. $  %By Lemma \ref{lem1}, we  can arrange that $\int_{F}\omega_X \gg Q$.
If $m<0$, then $E(\alpha + m[F])  <0$.  However, our $(X, \pi_X, \omega_X, J )$ is  non--negative,  thus we have $E(\mathcal{C}) \ge 0 $ for any $J$--holomorphic current. Therefore, the class $Z_{\alpha} + m[F]$ has no holomorphic representative.
\end{proof}

\begin{lemma}  \label{lem8}
Given a  positive integer $Q$,  if $\int_F \omega_X \ge Q+1$,  then  the conclusion in Lemma \ref{lem3} still holds  for any $\Omega_X$--tame almost complex structure $J$ which is sufficiently close to $ \mathcal{J}_h(X, \omega_X)$.
\end{lemma}
\begin{proof}
Given $\epsilon_0>0$, then for any $\Omega_X$--tame almost complex structure $J$ which is sufficiently close to $ \mathcal{J}_h(X, \omega_X)$,  we have   $E(\mathcal{C}) \ge -\epsilon_0$ for any $J$--holomorphic current.  Otherwise, we can find a sequence of $\Omega_X$--tame almost complex structures $\{J_n\}_{n=1}^{\infty}$ and $J_n$--holomorphic currents $\mathcal{C}_n$ such that $\{J_n\}_{n=1}^{\infty}  $ converges to $ J_{\infty} \in \mathcal{J}_h(X, \omega_X)$ and $E(\mathcal{C}_n)< -\epsilon_0$.
According to Taubes's Gromov compactness (see Lemma 6.8 of  \cite{HT}),  $\{\mathcal{C}_n\}_{n=1}^{\infty}$ converges to a $J_{\infty}$--holomorphic current (possibly broken) $\mathcal{C}_{\infty}$ in current sense. Therefore, $E(\mathcal{C}_{\infty}) \le -\epsilon_0<0$ which contradicts  $J_{\infty} \in \mathcal{J}_h(X, \omega_X)$.

The rest of the argument is the same as in Lemma \ref{lem3}.
\end{proof}

\subsection{Cobordism maps}
In this subsection, we show that   definition   (\ref{eq2}) makes sense. The proof   is essentially  the same as \cite{CG}. The argument there can use to show that if $(Y, \pi, \omega)$ only consists of $Q$--negative elliptic orbits and positive hyperbolic orbits, then the ECH index is non--negative for every holomorphic curve.  However, the conclusion is not true in general  if $(Y, \pi, \omega)$  also contains   $Q$--positive elliptic orbits.    But in our case,    the conclusion still holds  due to the explicit computation of the ECH index and Fredholm index.

%First of all, we give a formula for the Fredholm index of   holomorphic curves.

\begin{lemma}\label{lem6}
Let  $C \in \mathcal{M}^J(\alpha, Z_{\alpha} + m[F])$ be an irreducible holomorphic curve.  Let  $q=[\alpha] \cdot [F] \le Q$. Then  the Fredholm index of the holomorphic curve $C$ is
\begin{equation*}
 {\rm{ind}} C = 2g(C)-2 + h(C)  + 2q + 4m(1-g(F))+ 2  e_+(C),
\end{equation*}
where $e_+(C)$ is the number of ends at $Q$--positive elliptic orbits and $h(C)$ is the number of ends at hyperbolic orbits.
\end{lemma}
\begin{proof}
By Lemma \ref{lem10} and the Adjunction formula,  we have  $$c_{\tau}(Z_{\alpha} + m[F]) = q + mc_{1}([F]) =q+2m(1-g(F)). $$
The conclusion follows from the definition of the Fredholm index,  (\ref{eq6}) and  Lemma \ref{lem4}.
\end{proof}

Given  an irreducible simple  holomorphic curve $C \in \mathcal{M}^J(\alpha)$,  define $e_Q(C)$ to be the total multiplicity of all $Q$--negative elliptic orbits in $\alpha$. $C $ is a \emph{special holomorphic plane} if it  is a holomorphic plane with $I(C)=\rm{ind} C =0 $ and has exactly one positive  end at a simple $Q$--negative elliptic orbit. (See Definition 3.15 of \cite{CG}.)

In \cite{H5}, Hutchings defines a self--intersection number $C\star C$ for a simple holomorphic curve $C$. For its definition, please  see  the following lemma. Here $C \star C$ agrees with the usual self--intersection number  when  $C$ is closed. To ensure that the ECH index is non--negative, we need to show that this self--intersection is non--negative.
\begin{lemma} \label{lem5}
Let  $C \in \mathcal{M}^J(Z_{\alpha}+m[F])$ be a simple irreducible holomorphic curve and $q=[\alpha] \cdot [F] \le Q$. Suppose that $g(F)-1 \ge Q $ and $J$ is a generic  almost complex structure which is sufficiently close to $\mathcal{J}_h(X, \omega_X)$.  Then
\begin{equation*}
2C \star  C = 2g(C)-2 +  {\rm{ind} }C+ h(C) + 2e_Q(C) +  4\delta(C) \ge 0.
\end{equation*}
In particular, $I(\mathcal{C}) \ge 0$.  In addition,  if $ I(C)={\rm{ind}} C =0$, then $C \star C =0$ if and only if $C $ is a special holomorphic plane which satisfies  $q=e_Q(C)=1$ and $h(C)=e_+(C)=0$.
\end{lemma}
\begin{proof}
We argue by contradiction.  Assume that $C \star  C < 0$.  Then we must have $g(C)=e_Q(C)=\delta(C)=0$.

If  $h(C)=0$, then $ {\rm{ind} }C$ is even and hence  $ {\rm{ind} }C=0$.   According to Lemma \ref{lem6}, we have
\begin{equation*}
-1+e_+(C) +q+2m(1-g(F))=0.
\end{equation*}
On the other hand,  $e_+(C) +q+2m(1-g(F)) \le 2q- 2mQ $. Keep in mind that $m \ge 0$ because of Lemma \ref{lem8}.  As a result, $m=0$. Then  ${\rm{ind} }C= 2q-2 + 2e_+(C) =0$ implies that $e_+(C)=0$ and $q=1$.  Therefore,  $C$  is closed with degree $1$.  This is impossible.

If $h(C)=1$, then $ {\rm{ind} }C=0$; otherwise $C \star C \ge 0$.  However, by Lemma \ref{lem6}, we know that $ {\rm{ind} }C$ is odd. We get a  contradiction.

In conclusion, $C \star  C \ge 0$.  Follows from Proposition 4.8 of \cite{H5} and Theorem 4.15 of \cite{H2}, we have $I(\mathcal{C}) \ge 0$.

Now let us prove the last statement. Assume that $ {\rm{ind}} C =0$ and  $C \star  C =0$. Then we have $\delta(C)=0$ and $$2g(C)-2 + h(C) + 2e_Q(C)=0.$$ If $g(C)=0$,  then there are two possibilities:
\begin{enumerate}
\item
$ h(C)=0$  and $ e_Q(C)=1$;
\item
$h(C)=2$  and $ e_Q(C)=0.$
\end{enumerate}
On the other hand, $ {\rm{ind}} C =0$  and Lemma \ref{lem6} implies that  
\begin{equation} \label{eq22}
h(C)+ 2q+2e_+(C) +4m(1-g(F))=2.  
\end{equation}
 It is worth noting that $ h(C)+ 2q+2e_+(C) +4m(1-g(F)) \le 4q-4mQ$ due to our assumption.  In either cases, we have $m=0$; otherwise, the left hand side of (\ref{eq22}) is strictly less then $2$.

 In the first case that $h(C)=0$ and $ e_Q(C)=1$, then $2q+2e_+(C) =2$ implies that $q=1$ and $e_+(C)=0$.  Consequently,  $C$ is a special holomorphic plane with one positive end at a simple $Q$--negative periodic orbit and has no other ends.
 In the second case,  $q+e_+(C) =0  $ implies that $q=0$. This is impossible.
 %contradict with the assumption that $C$ is not closed.

If $g(C)=1$,  then $C\star C=0 $ implies that  $h(C)=e_Q(C)=0$.  By   $ {\rm{ind}} C =0$  and Lemma \ref{lem6},  we have  
\begin{equation} \label{eq23}
  2q+2e_+(C) +4m(1-g(F))=0.  
\end{equation}
Suppose that $Q<g(F)-1$. If $m \ge 1$, then the left hand side of (\ref{eq23}) is strictly less than zero.   If $m=0$, then $q=e_+(C)=0$. $C$ is  a closed  curve  and $[C]=n[F]$ because of Lemma \ref{lem2}.  The Adjunction formula and our assumption $g(F) \ge 2$ implies that $C$ cannot have genus $1$.

If $Q=g(F)-1$, then $C$ can not be closed as before.  Also,  $h(C)=e_Q(C)=0$ implies that all the ends of $C$ are asymptotic to $Q$--positive elliptic orbits. By Lemma \ref{lem4}, $I(C) \ge 2$ contradicts with $I(C)=0$.
 \end{proof}

\begin{definition} \label{def1}
%Let $\gamma $ be an embedded periodic orbit with degree $1$. Let $J \in \mathcal{J}_h(X, \omega_X)$ and $C \in \mathcal{M}^J(\gamma)$,
Let $u: \overline{D} \to \overline{X}$ be a section which is asymptotic to a periodic orbit with degree $1$.  The section  $u$ is called horizontal  if  $Im(du) \subset T\overline{X}^{hor}$.
\end{definition}
Note that for any $J \in \mathcal{J}_h (X, \omega_X)$, the horizontal section is $J$--holomorphic. % The horizontal section is rigidity in the following sense:  Recall that $(X, \omega_X, J) $ is nonnegative.   Given an orbit set $\alpha$ with degree $1$ and $ C \in \mathcal{M}^J(\alpha)$,  if  $E(C)=0$,  then  $C$ is a  horizontal section. (See  \cite{PS1}.)

%\begin{theorem} \label{thm2}
%Fix  an integer $Q $. For a generic  $\Omega_X$--tame  almost complex structure  $J$ which is sufficiently close to $ \in \mathcal{J}_h(X, \omega_X)$, then we have a   homomorphism
%\begin{equation*}
%PFH(X, \omega_X)_J : \mathds{A}(X) \otimes PFH_*(Y, \omega, Q)  \to \mathbb{Z}
%\end{equation*}
%defined by counting holomorphic curves,   where $\mathds{A}(X)=[U]\otimes \Lambda^*H_1(X, \mathbb{Z})$ and $U$ is the $U$--map (see \cite{H3}).
%\end{theorem}
\paragraph{Proof of the part \ref{A} of Theorem \ref{thm3}  }
\begin{proof}
To define the cobordism map, let us consider the moduli space $\mathcal{M}_{0}^J(\alpha + m[F])$ and $\mathcal{C}=\sum_k d_k C_k \in \mathcal{M}_{0}^J(\alpha + m[F])$.  Due to  Lemmas \ref{lem3} and \ref{lem8}, we always assume that  $m\ge0 $.

Let us first  consider the case that  $Q>g(F)-1$. Choose a generic almost complex structure $J \in \mathcal{J}_h(X, \omega_X)$.  Since $Q>g(F)-1$, the element $\mathcal{C} \in \mathcal{M}_0^J(\alpha, Z_{\alpha} + m[F]) $  cannot contain a closed component; otherwise $I\ge 2$. (See Lemma 5.7 of \cite{GHC}.)

 By Lemma \ref{lem4} and Theorem \ref{thm1} and (\ref{eq14}), $I(\mathcal{C}) =I_m(\alpha) \ge 0$ for any orbit  set and $m \ge 0$.  Moreover, $I=0$ if and only if  $m=0$ and the orbits set  is of the form  $\alpha=(\Pi_a e_a^{m_a})  e_0^{m_0} e_1^{m_1}$ , where $a $ is  a critical point of $f_S$ with $\nabla^2 f_S(a)>0$.  Thus let us assume $\alpha=(\Pi_a e_a^{m_a})  e_0^{m_0} e_1^{m_1}$ and $m=0$.   By the energy formula and Remark \ref{rmk2},  we have  $E(\alpha)=\sum_a f_S(a)$.

 We claim that the only element in $\mathcal{M}_0^J(\alpha)$ is of the form $\mathcal{C}=\sum_a m_a C_{e_a} + m_0 C_{e_0} +m_1 C_{e_1}$, where $C_{e_*}$ is a special holomorphic plane. %is horizontal section of $\pi_X: \overline{X} \to \overline{D}$.
 To see this, let  $\mathcal{C} =\sum_k d_k C_k$.  %By Lemma \ref{lem5}, we have $C_k \cdot C_k \ge 0$.
 First note  that all periodic orbits in $\alpha$ are $Q$--negative elliptic,   therefore  $C_k \star  C_k \ge 0$   by definition.
 Proposition 4.8 of \cite{H5} and Theorem 4.15 of \cite{H2}  imply that ${\rm{ind} } C_k=I(C_k)=0$ for any $k$.  By the fact that  $\pi_X : C_k \to \overline{D}$ is holomorphic and the  Riemann--Hurwitz formula,  we have 
 $$2g(C_k)-2=b-e_-(C_k)-q_k,$$
 where $q_k$ is the degree of $\pi_X \vert_{C_k}$, $e_-(C_k)$ is the number of ends of $C_k$ and $b \ge 0$ is  a certain counting    the branched points.  According to  Lemma \ref{lem6}, we have   $$0={\rm{ind} } C_k= 2g(C_k)-2 +2q_k =  b+ q_k -e_-(C_k).$$  Consequently, $b=g(C_k)=0$ and $q_k=e_-(C_k)=1$.  Thus $C_k$ is  a special holomorphic plane.  By Remark \ref{rmk3}, $C_k$ actually is  a  horizontal section because  $E(C_k)=0$ or $f_S(a)$.
 %According to  Lemma \ref{lem6}, ${\rm{ind}} C_{e_*} =0$.

 Let $\tilde{C} $ be an unbranched cover of $C_{e_*}$  (disjoint copies of $C_{e_*}$) with ${\rm{ind}}\tilde{C}=0$. It satisfies ${\rm ind} \tilde{C}  > 2g(\tilde{C} )-2 + h_+(\tilde{C} )$, where $h_+(\tilde{C})$ is the number of ends at even periodic orbits ($h_+(C)=0$ in our case).  
   By the automatic transversality theorem \cite{Wen},   $\tilde{C} $  is Fredholm regular for any  $J \in  \mathcal{J}_h(X, \omega_X)$.   Therefore,  $\mathcal{M}_0^J(\alpha)$  only  consists of a single element and  hence it is compact  automatically.
   Then we define $PFC(X, \omega_X)_J \alpha=\# \mathcal{M}_0^J(\alpha) $;  otherwise,  $PFC(X, \omega_X)_J =0 $. %For any $J'$ sufficiently close to $\mathcal{J}_h(X, \omega_X)$, the same conclusion still hold.

To show that it is a chain map, i.e., $PFC(X, \omega_X )_J \circ \partial =0 $. %it suffices to show that $PFC(X, \omega_X ) \circ \partial \beta=0  $ for any   ECH generator $\beta$ satisfying $<\partial \beta, \alpha > \ne 0$.
Let  $\beta$ be an ECH generator and $\mathcal{C}'=\sum_k d'_kC'_k \in \mathcal{M}_1^J(\beta + m[F])$ be a  holomorphic current with ECH index 1.   According to  Lemma \ref{lem4} and Theorem \ref{thm1} and (\ref{eq14}), $I_m(\beta)=1$ implies that $m=0$ and $\beta$ doesn't involve $Q$--positive elliptic orbits. By definition, $C_k' \star C'_k \ge 0$.  In this case, one can show that $\mathcal{C}' $ consists of embedded holomorphic curves and special holomorphic planes by using the ECH inequalities in \cite{H2} and \cite{H5}.  Thus the moduli space $\mathcal{M}_1^J(\beta)$ is a 1--dimensional manifold.  %one can show that $C_k \cdot C_k \ge 0$ and equality holds if and only if $C_k$ is the special holomorphic plane with ${\rm{ind}} C_k=h(C_k)=e_Q(C_k)-1=0$. (Cf. \cite{CG})
Note that  the ECH index of holomorphic curves are non--negative by Lemma \ref{lem4} and Theorem \ref{thm1} and (\ref{eq14}). Therefore, a broken holomorphic curve arising as a limit of curves in  $\mathcal{M}_1^J(\beta)$ consists of a holomorphic curve with $I=1$ in the top level, a curve with $I=0$ in the cobordism level, and connectors (branched covers of trivial cylinders) in between.  Make use of the gluing analysis in \cite{HT1} and \cite{HT2},   we have $PFC(X, \omega_X )_J \circ \partial \beta =0 $.

Now we turn to consider the case that  $Q \le g(F)-1$.  Fix  a generic $\Omega_X$--tame almost complex $J$ structure which is sufficiently close  to $ \mathcal{J}_h(X, \omega_X)$.  By Lemma \ref{lem2}, we know that  a closed holomorphic curve has homology class $m[F]$ for $m \ge 0$.  However, $I(m[F])<0$, this is impossible for a generic $J$.   Hence,  there are no closed components  in $\mathcal{C}.$
 According to Proposition  4.8 of \cite{H5} and Corollary  \ref{lem5}, $I(\mathcal{C})=0$ implies that
\begin{enumerate}
\item
$I(C_k)={\rm{ind}} C_k =0$ and $C_k$ is embedded for each $k$.
\item
$C_k \star C_l=0$ for   $k \ne l$.
\item
\begin{enumerate}
\item
If $C_k \star C_k  >0$, then $d_k=1$.
\item
If $C_k \star C_k =0$, then $C_k$ is a  special  holomorphic plane which is asymptotic to a simple $Q$--negative periodic orbit with degree $1$.
\end{enumerate}
\end{enumerate}
Therefore, $\mathcal{C} = \mathcal{C}_0+ \sum_a d_a C_{e_a} + d_1C_{e_1} + d_0 C_{e_0}$, where $\mathcal{C}_0$ is an embedded holomorphic current with $I=0$ and    $C_{e_*} $ are special  holomorphic planes. % which is asymptotic to $e_*$ and $e_*$ is $Q$-negative elliptic orbit.
 One can use the  same argument in \cite{CG} (Proposition 3.16) to show that $\mathcal{M}_{0}^J(\alpha )$ is a finite  set. As a consequence,  the definition  (\ref{eq2}) makes  sense in this case.
To show that   $PFC(X, \omega_X )_J $  is a chain map, the argument is the same as the case $Q>g(F)-1$. We left the details to  readers. 
%   Again,  make use of the gluing analysis in \cite{HT1} and \cite{HT2},  we have  $PFC(X, \omega_X )_J \circ \partial =0 $.
 %he main different with  \cite{GHC}   is that the holomorphic current could be multiple cover, according to our previous discussion, lucky the curve with multiple cover has to be unbranched cover of  holomorphic plane  $C_{e_*}$.  We can deal with this issue by using the same argument in \cite{CG}.

 Finally, in either cases, it is not difficult  to extend the cobordism maps  over the  ring $\mathds{A}(X)$. For more details, please see \cite{CG}.

\end{proof}

\paragraph{Proof of the  part \ref{B} of Theorem \ref{thm3}}
\begin{proof} %[Proof of Theorem \ref{thm3}]

\begin{comment}
According to \ref{eq14}, the ECH index could be negative prior.  However,  note that $I<0$ only when $m<0$.   According to Lemma \ref{lem3}, there is no holomorphic current with $I<0$.

When $m>0$, then by \ref{eq14},  $I>0$ and hence $\mathcal{M}_0^J(\alpha, {Z}_{\alpha}+m[\Sigma])= \emptyset$.
\end{comment}

\begin{itemize}
\item
\begin{comment}
Let us first consider the case that  $Q>g(F)-1$. Choose a generic almost complex structure $J \in \mathcal{J}_h(X, \omega_X)$.
Consider the moduli space $\mathcal{M}_0^J(\alpha, Z_{\alpha} + m[F]) $.  By lemma \ref{lem3}, we may assume that $m\ge0 $.   Since $Q>g(F)-1$, the element $\mathcal{C} \in \mathcal{M}_0^J(\alpha, Z_{\alpha} + m[F]) $  cannot contain closed component, otherwise $I>0$.

    By Lemma \ref{lem4} and Theorem \ref{thm1} and \ref{eq14}, $ I_0(\alpha)=0$ if and only if  $m=0$ and the orbits set of the form $\alpha=\sum_a m_a e_a +  m_0 e_0 + m_1 e_1$, where $a $ is a critical point of $f$ with $\nabla^2 f(a)>0$.  Thus let us assume $\alpha=\sum_a m_a e_a +  m_0 e_0 + m_1 e_1$ and $m=0$.   Since $E(\alpha)=0$, the only element in $\mathcal{M}_0^J(\alpha)$ is of the form $\mathcal{C}=\sum_a m_a C_{e_a} + m_0 C_{e_0} +m_1 C_{e_1}$, where $C_{e_*}$ are horizontal section of $\pi_X: X \to D$.  According to our Lemma \ref{lem6}, ${\rm{ind}} C_{e_*} =0$ and satisfies the inequality ${\rm ind} C > 2g(C)-2 + h_+(C)$. %Make use of the same argument in \cite{PS1}, $C_{e_a}$ is Fredholm regular.  Alternative, ${\rm{ind} }C_{e_a}> 2g(F)-2+h(C_{e_a}) $, thus
By the automatically regular theorem \cite{Wen}, $C_{e_*}$  are super-rigid,    thus the transversality can be achieved for any  $J \in  \mathcal{J}_h(X, \omega_X)$. Therefore,  $\mathcal{M}_0^J(\alpha)$  consists of only one element and hence $\# \mathcal{M}_0^J(\alpha)=1$.
In conclusion,  $PFC(X, \omega_X)_J \alpha =1$.
\end{comment}
Assume that $Q>g(F)-1$.  According to  the proof of    part \textbf{A}, we know that  $PFC(X, \omega_X)_J \alpha=0$ unless that    $\alpha=(\Pi_a e_a^{m_a})  e_0^{m_0} e_1^{m_1}$, where $a$ are    critical points  of $f_S$ with $\nabla^2 f_S(a)>0$. Moreover,  there is only one element  in $\mathcal{M}_0^J(\alpha)$ which consists of copies of horizontal sections. According to \cite{CG}, the sign of such an element is positive.  In sum, $ \# \mathcal{M}_0^J(\alpha)=1$.

\item
Now let us consider the second case that $2Q+1\le g(F)$. %Without loss of generality,  we may assume that  there is no critical point $\{a\}$ such that $\nabla^2f(a)>0$.
Assume that the Morse function $f_S: S \to \mathbb{R}$ is chosen as in the proof of Theorem 5.3 of \cite{H1}. %Then $f_S$ has no critical points $a$ such that $\nabla^2f_S(a)>0$.
Then $(Y, \pi, \omega)$ has no $Q$--negative elliptic orbits with degree 1 except $e_0, e_1$.
By Theorem 5.3 of  \cite{H1} and our assumption, the generators of $PFH_*(Y, \omega, Q)$ are   represented by   a  linear combination of orbits sets  $\alpha =  (\Pi_a  h_a)   \alpha_0$, where   $a$  are  saddle points of $f_S$.    Let us  denote  the degree of   $\Pi_a  h_a$  by    $q_h$.   Then  $\alpha_0$ is an orbit  set of index $0 \le I_0(\alpha_0)  \le 2Q'-1$ and $Q'=Q -q_h$.

\begin{comment}
Let us explain a  little bit  about the generators here. In \cite{H1}, Hutchings splits the  differential $\partial= \partial_0 + \partial_1 + \cdots$ according to the wrapping number. Here $\alpha_0$ is given by the generators in Lemma 4.9 of \cite{H1}. Then $\partial_0 \alpha =0$.  We claim that $\partial_i\alpha=0$ for $i \ge 1$. Otherwise, there exists a holomorphic curve $C$ from $\alpha$ to $\beta$ with $I(C)=1$ and $\eta(C) \ge 1$, where $\eta(C)$ is the wrapping number.  Then 
\begin{equation} \label{eq24}
	1=I(C)=I_0(\alpha) -I_0( \beta) + 2\eta(C)(Q+1-g(F))
\end{equation}
implies that $I_0(\beta) <0$, contradicts with Lemma \ref{lem4} and Theorem \ref{thm1}.   Therefore, $\alpha$ is a cycle.   We claim that $\alpha$ is not exact. Otherwise we have $\alpha=\partial \beta =\partial_0 \beta + \partial_{\ge 1} \beta$. If $\partial_{\ge 1} \beta \ne 0$, then \ref{eq24} implies that $I_0(\beta) \ge 2Q +1$. Compare the $I_0$ index of $\alpha$ and $\partial \beta$, we know that $\partial_0 \beta =0$.  Since  $H_*(\partial_0)$ is supported in $0\le * \le 2Q-1$, we have $[\beta] =0 \in H_*(\partial_0)$, i.e., $\beta=\partial_0 \gamma$ for some chain $\gamma$. As a result, we have $\alpha=\partial_{\ge1} \partial_0 \gamma =\partial_0 \partial_{\ge1}  \gamma $. Contradicts with the fact that $\alpha$ is a generator of $H_*(\partial_0)$.  
\end{comment}

By Lemma \ref{lem4} and Theorem \ref{thm1} and (\ref{eq14}),
\begin{equation*}
I_m(\alpha)=q_h + I_0(\alpha_0)+ 2m(Q+1-g(F)).
\end{equation*}
Since  $2Q+1\le g(F)$, $I_m(\alpha) \le q_h +I_0(\alpha_0) -2m Q < 2Q -2mQ$. According to Lemma \ref{lem8}, we may assume that $m\ge 0$.  Therefore, $I_m(\alpha)<0$ unless $m=0$.
When $m=0$,  $I_0(\alpha)=0$ if and only if $q_+=q_h=I_0(\alpha_0)=0$.  In particular,  $\alpha=  e_0^{m_0} e_1^{m_1}=e^{m_0+m_1}$.

Recall that for all $J \in  \mathcal{J}_h(X, \omega_X)$,  transversality of $\mathcal{M}^J_0(e_0^{m_0} e_1^{m_1})$ can be achieved.  If    $J \in \mathcal{J}_{tame}(X, \omega_X)$   is sufficiently close to $  \mathcal{J}_h(X, \omega_X)$, then $\mathcal{M}_0^J(e_0^{m_0} e_1^{m_1})$ also consists of a single  element with a  positive sign.

\end{itemize}
\end{proof}

School of Mathematical Sciences, University of Adelaide
Adelaide SA, Australia

\verb| E-mail address:  guanheng.chen@adelaide.edu.au|

\end{document}